\documentclass[12pt,reqno]{amsproc}

\usepackage{times}

\usepackage[inline]{enumitem}
\usepackage{amsmath,amsthm,amssymb,wasysym,mathtools,array}

\allowdisplaybreaks
\numberwithin{equation}{section}

\usepackage{float}
\usepackage{graphicx}
\usepackage{caption}
\usepackage{subcaption}
\usepackage{xcolor}

\usepackage{tikz}
\usepackage{pgfplots}
\pgfplotsset{compat=1.18}

\newcommand{\D}{\mathbb{D}}

\newtheoremstyle{fancytheorem}
{.5\baselineskip±.2\baselineskip}
{.5\baselineskip±.2\baselineskip}
{\itshape\addtolength{\leftskip}{0mm}\setlength{\parindent}{0em}}
{0mm}
{\bfseries}
{}
{.5em}
{\thmname{#1}\thmnumber{ #2}. \thmnote{[#3]}}

\theoremstyle{fancytheorem}

\newtheorem{theorem}{Theorem}[section]
\newtheorem{lemma}[theorem]{Lemma}
\newtheorem{corollary}[theorem]{Corollary}
\newtheorem{remark}[theorem]{Remark}

\theoremstyle{fancydefinition}
\newtheorem{definition}[theorem]{Definition}

\pgfmathdeclarefunction{den}{1}{%
  \pgfmathparse{
    (4 + 2*sqrt(2))
    - (2*sqrt(2) + 2)*cos(#1)
  }%
}

\pgfmathdeclarefunction{lemMag}{2}{%
  \pgfmathparse{
    sqrt(1 + 2*(#1)*cos(#2) + (#1)^2)
  }%
}

\pgfmathdeclarefunction{lemX}{2}{%
  \pgfmathparse{
    sqrt(
      (lemMag(#1,#2) + 1 + (#1)*cos(#2))/2
    )
  }%
}

\pgfmathdeclarefunction{lemY}{2}{%
  \pgfmathparse{
    sign(sin(#2)) *
    sqrt(
      max(
        0,
        (lemMag(#1,#2) - 1 - (#1)*cos(#2))/2
      )
    )
  }%
}

\pgfmathdeclarefunction{luneMag}{2}{%
  \pgfmathparse{
    sqrt(
      1 + 2*(#1)^2*cos(2*#2) + (#1)^4
    )
  }%
}

\pgfmathdeclarefunction{luneX}{2}{%
  \pgfmathparse{
    (#1)*cos(#2)
    + sqrt(
      (luneMag(#1,#2)
      + 1
      + (#1)^2*cos(2*#2))/2
    )
  }%
}

\pgfmathdeclarefunction{luneY}{2}{%
  \pgfmathparse{
    (#1)*sin(#2)
    + sign(sin(2*#2)) *
    sqrt(
      max(
        0,
        (luneMag(#1,#2)
        - 1
        - (#1)^2*cos(2*#2))/2
      )
    )
  }%
}

\pgfmathdeclarefunction{quadX}{2}{%
  \pgfmathparse{
    1
    + sqrt(2)*(#1)*cos(#2)
    + ((#1)^2/2)*cos(2*#2)
  }%
}

\pgfmathdeclarefunction{quadY}{2}{%
  \pgfmathparse{
    sqrt(2)*(#1)*sin(#2)
    + ((#1)^2/2)*sin(2*#2)
  }%
}

\usepackage{hyperref}
\hypersetup{hidelinks}

\title{Sharp Subordination Radii through Boundary Minimization}

\author[P. Poomalai]{Palanisamy Poomalai }
\address{Department of Mathematics \\National Institute of Technology\\Tiruchirappalli-620015,  India }
\email{poomalainitt@gmail.com}

\author[V. Ravichandran]{Vaithiyanathan Ravichandran }
\address{Department of Mathematics \\National Institute of Technology\\Tiruchirappalli-620015,  India }
\email{vravi68@gmail.com; ravic@nitt.edu}



\begin{document}

\begin{abstract}
Let $\varphi$ and $\psi$ be univalent functions in the unit disk
$\mathbb D$ satisfying $\varphi(0)=\psi(0)=1$ and
$\varphi\not\prec\psi$, and let $\mathcal P_{\varphi}$ denote the
class of analytic functions $p$ in $\mathbb D$ such that
$p\prec\varphi$. Under suitable analytic continuation, univalence,
and boundary-behavior hypotheses on the branch of $\varphi^{-1}$
satisfying $\varphi^{-1}(1)=0$, we determine the sharp
$\mathcal P_{\psi}$-radius of $\mathcal P_{\varphi}$ as
\[
\mathcal R(\varphi,\psi)
=
\min_{|\zeta|=1}
\left|
\varphi^{-1}\bigl(\psi(\zeta)\bigr)
\right|.
\]
For Ma--Minda functions satisfying these hypotheses, the same radius
is sharp for the associated classes $\mathcal{ST}(\varphi)$ and
$\mathcal{CV}(\varphi)$ with respect to
$\mathcal{ST}(\psi)$ and $\mathcal{CV}(\psi)$, respectively.
We apply the general result to
$
\varphi_{\mathrm L}(z)=\sqrt{1+z},
\varphi_{\leftmoon}(z)=z+\sqrt{1+z^2},
\varphi_{\mathrm{Lim}}(z)=\left(1+{z}/{\sqrt2}\right)^2
$,
for several choices of the Ma--Minda function $\psi$. The resulting boundary minima are evaluated analytically.
\end{abstract}

 \maketitle
 
\section{Introduction}

Let $\mathcal A$ denote the class of functions $f$ analytic in the
open unit disk
$\mathbb D:=\{z\in\mathbb C:|z|<1\},$
normalized by $f(0)=0$ and $f'(0)=1$. For analytic functions $f$ and
$g$ in $\mathbb D$, we say that $f$ is subordinate to $g$, written as
$f\prec g$, if there exists an analytic function
$\omega:\mathbb D\to\mathbb D$ satisfying $\omega(0)=0$ such that
$f(z)=g(\omega(z))$\cite{MR0708494}. If $g$ is univalent in $\mathbb D$, then
\[
f\prec g
\quad\Longleftrightarrow\quad
f(0)=g(0)
\quad\text{and}\quad
f(\mathbb D)\subset g(\mathbb D).
\]

A function $\varphi$ analytic and univalent in $\mathbb D$,
satisfying $\varphi(0)=1$, $\varphi'(0)>0$, and
$\operatorname{Re}\varphi(z)>0$ in $\mathbb D$, and mapping
$\mathbb D$ onto a domain symmetric with respect to the real axis
and starlike with respect to $1$, is called a
\emph{Ma--Minda function} \cite{MR1343506}. We denote by
$\mathcal P_{\varphi}$ the class of analytic functions $p$ in
$\mathbb D$ satisfying $p\prec\varphi$.

When $\varphi$ is a Ma--Minda function, the associated starlike and
convex classes are defined, for $f\in\mathcal A$, by
\[
f\in\mathcal{ST}(\varphi)
\quad\Longleftrightarrow\quad
\frac{zf'(z)}{f(z)}\prec\varphi(z)
\;
\text{and}
\;
f\in\mathcal{CV}(\varphi)
\quad\Longleftrightarrow\quad
1+\frac{zf''(z)}{f'(z)}\prec\varphi(z),
\]
where the defining expressions are understood by analytic
continuation at the origin. These classes provide a common framework
for many subclasses of univalent functions associated with geometric
domains.

Ma--Minda classes associated with various geometric regions have
been studied extensively. These include the classes corresponding to
the lemniscate \cite{MR1473947,MR3902942}, lune
\cite{MR3419845,MR3718233}, exponential function
\cite{MR3394060}, cardioid \cite{MR3536076}, rational function
\cite{MR3496681}, lima\c{c}on
\cite{Masih,MR4386722,MR4478088}, nephroid
\cite{MR4190740,MR4129627}, three-leaf function
\cite{MR4099174}, hyperbolic cosine function
\cite{MR4620985}, and modified sigmoid function
\cite{MR4044913}; inclusion relations and radius problems for these
classes appear, for example, in
\cite{MR4303808,MR4129627,MR4510346,MR4386722,MR4216363,DevadigaEtAl2026Atlas,GurumurthyEtAl2026Certified}.
The inverse-composition approach relevant to the present work
appears in Ali et al.~\cite[Theorem~2.2]{MR2953523}. For two
Janowski functions \cite{MR0328059} $P$ and $Q$, they used the composition
$P^{-1}\circ Q$ to determine the radius between the corresponding
Janowski starlike classes. Subordination-based radius problems were
subsequently studied by Sebastian and Ravichandran
\cite{MR4568286}, who obtained radius estimates for several classes
of analytic functions.

A typical radius problem asks for the largest $r\in(0,1]$ for which
the dilation of every function from one class belongs to another
prescribed class. Here we consider the $\mathcal P_{\psi}$-radius of
$\mathcal P_{\varphi}$, denoted by $\mathcal R(\varphi,\psi)$: the
largest $r\in(0,1]$ such that $p(rz)\prec\psi(z)$ for every
$p\in\mathcal P_{\varphi}$. Since $p\prec\varphi$, this condition
is equivalent to the geometric inclusion
$
\varphi(r\mathbb D)\subset\psi(\mathbb D).
$
Our approach is based on the inverse mapping of the source function.
Suppose that the branch of $\varphi^{-1}$ satisfying
$\varphi^{-1}(1)=0$ extends analytically and univalently to a domain
containing
$
\varphi(\mathbb D)\cup\psi(\mathbb D),
$
and that $\varphi^{-1}\circ\psi$ extends continuously to
$\overline{\mathbb D}$. Under these hypotheses, when
$\varphi\not\prec\psi$, we establish the formula
\[
\mathcal R(\varphi,\psi)
=
\min_{|\zeta|=1}
\left|\varphi^{-1}\bigl(\psi(\zeta)\bigr)\right|.
\]
This formula reduces the radius problem to the evaluation of a
boundary minimum; the main analytical task in each application is
to determine this minimum exactly and to establish the sharpness of
the resulting radius.

We emphasize that Theorem~\ref{thm:general-radius} is a reformulation:
the description of the radius through $\operatorname{dist}(0,\partial\Omega)$
with $\Omega=\varphi^{-1}(\psi(\mathbb D))$ underlies earlier work. What the
present paper adds is a precise statement under explicit hypotheses, including
a variant for boundary poles, the passage to $\mathcal{ST}(\varphi)$ and
$\mathcal{CV}(\varphi)$, and the exact evaluation of the boundary minimum for
the lemniscate, lune and lima\c{c}on source functions. Table~\ref{tab:known}
lists, for every pair treated in Sections~3--5, the radius and whether it is
new or known. The known radii are recovered here, and the remaining ones are,
to our knowledge, new. Among them is the sharp value
$\mathcal R(\varphi_{\leftmoon},\psi_{\mathrm{Ne}})=8/15$, which answers a
question left open by Wani and Swaminathan~\cite{MR4129627}.
We apply the radius formula to the three source functions
\[
\varphi_{\mathrm L}(z)=\sqrt{1+z},
\qquad
\varphi_{\leftmoon}(z)=z+\sqrt{1+z^2},
\qquad
\varphi_{\mathrm{Lim}}(z)
=\left(1+\frac{z}{\sqrt2}\right)^2,
\]
which are associated, respectively, with the lemniscate, lune, and
lima\c{c}on-type domains. The target functions considered in the
resulting radius problems include
\[
\psi_R(z)
=
1+\frac{z^2+kz}{k^2-kz},
\quad
k=\sqrt2+1,
\quad
\psi_{car}(z)=1+z+\frac{z^2}{2},
\quad
\psi_e(z)=e^z.
\]
\[
\psi_{\varrho}(z)=\cosh\sqrt z,
\quad
\psi_C(z)
=1+\frac{4z}{3}+\frac{2z^2}{3},
\qquad
\psi_{\mathrm{Ne}}(z)
=1+z-\frac{z^3}{3},
\]
\[
\psi_{3L}(z)
=1+\frac{4z}{5}+\frac{z^4}{5},
\qquad
\psi_{SG}(z)
=\frac{2}{1+e^{-z}},
\qquad
\psi_c(z)=1+ze^z.
\]
as well as the lemniscate, lune, and lima\c{c}on functions
appearing above.
The boundary minima are established analytically using polynomial factorization, positivity arguments, monotonicity, Sturm's theorem, and modulus estimates. The corresponding sharp radii for the associated Ma--Minda starlike and convex classes follow from the general radius result.

We establish the sharp $\mathcal{ST}(\psi_{\mathrm{Ne}})$-radius of
$\mathcal{ST}(\varphi_{\leftmoon})$. Wani and Swaminathan~\cite{MR4129627}
obtained the lower bound
\[
\frac{\sqrt{17}-1}{6}
\]
for this radius. We prove that the sharp value is
\[
\mathcal R(\varphi_{\leftmoon},\psi_{\mathrm{Ne}})
=\frac{8}{15}.
\]

Some of the radii obtained in this way coincide with results
in the literature, which they thereby recover; Table~\ref{tab:known} lists,
for every pair treated in Sections~3--5, the radius and whether it is new or known.

\begin{table}[htbp]
\centering
\small
\renewcommand{\arraystretch}{1.25}
\begin{tabular}{@{}>{\centering\arraybackslash}m{2cm}%
                  >{\centering\arraybackslash}m{2cm}%
                  >{\centering\arraybackslash}m{3.5cm}%
                  >{\centering\arraybackslash}m{3.2cm}@{}}
\hline
$\varphi$ & $\psi$ & $\mathcal R(\varphi,\psi)$ & Status \\
\hline
$\varphi_{\mathrm L}$ & $\psi_e$ & $1-e^{-2}$ & new \\
$\varphi_{\mathrm L}$ & $\psi_{\leftmoon}$ & $2(\sqrt2-1)$ & new \\
$\varphi_{\mathrm L}$ & $\psi_{\varrho}$ & $\sin^2 1$ & new \\
$\varphi_{\mathrm L}$ & $\psi_C$ & $8/9$ & known \cite{MR3536076} \\
$\varphi_{\mathrm L}$ & $\psi_{car}$ & $3/4$ & known \cite{MR4303808} \\
$\varphi_{\mathrm L}$ & $\psi_{\mathrm{Ne}}$ & $8/9$ & known \cite{MR4129627} \\
$\varphi_{\mathrm L}$ & $\psi_{3L}$ & $21/25$ & known \cite{MR4510346} \\
$\varphi_{\mathrm L}$ & $\psi_{\mathrm{Lim}}$ & $3\sqrt2-13/4$ & known \cite{MR4386722} \\
$\varphi_{\mathrm L}$ & $\psi_R$ & $8\sqrt2-11$ & known \cite{MR3496681} \\
$\varphi_{\mathrm L}$ & $\psi_c$ & $2e^{-1}-e^{-2}$ & known \cite{MR4216363} \\
\hline
$\varphi_{\leftmoon}$ & $\psi_{\mathrm{Ne}}$ & $8/15$ & new (improves \cite{MR4129627}) \\
$\varphi_{\leftmoon}$ & $\psi_R$ & $(5-3\sqrt2)/4$ & new \\
$\varphi_{\leftmoon}$ & $\psi_{car}$ & $3/4$ & new \\
$\varphi_{\leftmoon}$ & $\psi_{3L}$ & $3\sqrt{139}/70$ & known \cite{MR4510346} \\
$\varphi_{\leftmoon}$ & $\psi_e$ & $\sin 1$ & known \cite{DevadigaEtAl2026Atlas} \\
$\varphi_{\leftmoon}$ & $\psi_{\mathrm L}$ & $1/(2\sqrt2)$ & known \cite{DevadigaEtAl2026Atlas} \\
\hline
$\varphi_{\mathrm{Lim}}$ & $\psi_C$ & $\sqrt2\,(1-1/\sqrt3)$ & new \\
$\varphi_{\mathrm{Lim}}$ & $\psi_{\mathrm{Ne}}$ & $\sqrt2\,(\sqrt{5/3}-1)$ & new \\
$\varphi_{\mathrm{Lim}}$ & $\psi_R$ & $\sqrt2\,(1-\sqrt{2\sqrt2-2})$ & new \\
$\varphi_{\mathrm{Lim}}$ & $\psi_e$ & $\sqrt2\,(1-1/\sqrt e)$ & new \\
$\varphi_{\mathrm{Lim}}$ & $\psi_{3L}$ & $\sqrt{(24-6\sqrt{15})/5}$ & new \\
$\varphi_{\mathrm{Lim}}$ & $\psi_{SG}$ & $\sqrt2\,(\sqrt{2e/(e+1)}-1)$ & new \\
$\varphi_{\mathrm{Lim}}$ & $\psi_{car}$ & $\sqrt2-1$ & known \cite{MR4303808} \\
\hline
\end{tabular}
\caption{The sharp radii $\mathcal R(\varphi,\psi)$ obtained in Sections~3--5.}
\label{tab:known}
\end{table}

The paper is organized as follows. Section~2 presents the radius
formula, its consequences for the associated Ma--Minda starlike and
convex classes, and sufficient conditions for the inverse
hypotheses. Section~3 establishes ten sharp radius theorems for the
lemniscate-associated source function. Section~4 contains six sharp
radius theorems for the lune-associated source function, while
Section~5 establishes seven sharp radius theorems for the
lima\c{c}on-associated source function. Geometric illustrations of the domains corresponding to the functions
considered in Sections~3--5 are also provided. The final section contains concluding
remarks.

\section{A General Subordination Radius Theorem}

\begin{definition}
Let $\varphi$ and $\psi$ be univalent in $\D$ with
$\varphi(0)=\psi(0)=1$. The $\mathcal P_{\psi}$-radius of
$\mathcal P_{\varphi}$ is defined by
$
\mathcal R(\varphi,\psi)
=
\sup\left\{
r\in(0,1]:
p(rz)\prec\psi(z)
\text{ for every }p\in\mathcal P_{\varphi}
\right\}.
$
\end{definition}

The following theorem gives a formula for the $\mathcal P_{\psi}$-radius
in terms of the inverse function $\varphi^{-1}$, under suitable
analyticity and univalence assumptions.

\begin{theorem}\label{thm:general-radius}
Let $\varphi$ and $\psi$ be univalent functions in $\D$ satisfying
$\varphi(0)=\psi(0)=1$ and $\varphi\not\prec\psi$. Suppose that the
branch of $\varphi^{-1}$ satisfying $\varphi^{-1}(1)=0$ extends
analytically and univalently to a domain containing
$\varphi(\D)\cup\psi(\D)$, and that $\varphi^{-1}\circ\psi$
extends continuously to $\overline{\D}$.

If $p$ is analytic in $\D$ and satisfies $p\prec\varphi$, then
$p(rz)\prec\psi(z)$ for every
$0\le r\le\mathcal R(\varphi,\psi)$, where
\[
\mathcal R(\varphi,\psi)
=
\min_{|z|=1}
\left|\varphi^{-1}(\psi(z))\right|.
\]
The radius $\mathcal R(\varphi,\psi)$ is best possible.
\end{theorem}

\begin{proof}
The case $r=0$ is trivial since $p(0)=\psi(0)=1$, so let $0<r\le 1$. 
If $p\prec\varphi$, there exists a Schwarz function $\omega$ with $\omega(0)=0$ 
such that $p=\varphi\circ\omega$. By Schwarz's lemma, $\omega(r\D)\subset r\D$, 
and hence $p(r\D)=\varphi(\omega(r\D))\subset\varphi(r\D)$. Since $\psi$ is 
univalent and $p(0)=\psi(0)=1$, the inclusion $\varphi(r\D)\subset\psi(\D)$ 
implies $p(rz)\prec\psi(z)$. Conversely, taking $p=\varphi$ shows that any radius 
valid for all $p\in\mathcal P_{\varphi}$ must satisfy $\varphi(r\D)\subset\psi(\D)$. 
Therefore,
\[
\mathcal R(\varphi,\psi)
=
\sup\left\{r\in(0,1]: \varphi(r\D)\subset\psi(\D)\right\}.
\]

Let $F=\varphi^{-1}\circ\psi$ and $\Omega=F(\D)=\varphi^{-1}(\psi(\D))$. Since the 
branch of $\varphi^{-1}$ with $\varphi^{-1}(1)=0$ is analytic and univalent on a 
domain containing $\varphi(\D)\cup\psi(\D)$, it is injective there and satisfies 
$\varphi^{-1}(\varphi(z))=z$ for $z\in\D$. Consequently, for $z\in\D$, we have 
$\varphi(z)\in\psi(\D)$ if and only if $z\in\Omega$, and thus
\[
\varphi(r\D)\subset\psi(\D)
\quad\Longleftrightarrow\quad
r\D\subset\Omega.
\]
Since $F$ is univalent in $\D$, the set $\Omega$ is open and contains $0=F(0)$. 
Moreover, since $\varphi\not\prec\psi$, there exists $a\in\D$ with 
$\varphi(a)\notin\psi(\D)$, so $a\notin\Omega$. Setting 
$d=\operatorname{dist}(0,\partial\Omega)$, we have $0<d\le|a|<1$. Since $d\D$ is 
connected, contains $0\in\Omega$, and does not meet $\partial\Omega$, we have 
$d\D\subset\Omega$. On the other hand, $\partial\Omega$ is closed and nonempty, 
so it contains a point $w_0$ with $|w_0|=d$; since $\Omega$ is open, 
$w_0\notin\Omega$, and hence $r\D\not\subset\Omega$ for $r>d$. Thus the supremum 
is attained and $\mathcal R(\varphi,\psi)=d$.

It remains to express $d$ in terms of $F$, where we also write $F$ for its 
continuous extension to $\overline{\D}$. By the compactness of $\overline{\D}$ and 
continuity of $F$, we have $\overline{\Omega}=F(\overline{\D})$, so 
$\partial\Omega=F(\overline{\D})\setminus F(\D)\subset F(\partial\D)$. Conversely, 
if $|\zeta|=1$, then $F(\zeta)\in\overline{\Omega}$. If $F(\zeta)\in\Omega$, then 
$F(\zeta)=F(z_0)$ for some $z_0\in\D$. Since $F^{-1}:\Omega\to\D$ is continuous, 
letting $z\to\zeta$ within $\D$ gives $F(z)\to F(z_0)$ and hence $z\to z_0$, 
contradicting $|\zeta|=1>|z_0|$. Thus $F(\partial\D)\subset\partial\Omega$, so 
$\partial\Omega=F(\partial\D)$ and
\[
\mathcal R(\varphi,\psi)
=
\operatorname{dist}(0,\partial\Omega)
=
\min_{|\zeta|=1}|F(\zeta)|
=
\min_{|\zeta|=1}\left|\varphi^{-1}(\psi(\zeta))\right|,
\]
where the last expression denotes the continuous boundary extension of 
$\varphi^{-1}\circ\psi$. Finally, for $\mathcal R(\varphi,\psi)<r\le 1$, we have 
$\varphi(r\D)\not\subset\psi(\D)$, so $p=\varphi$ satisfies 
$p(rz)\not\prec\psi(z)$, showing that $\mathcal R(\varphi,\psi)$ is best possible.
\end{proof}

\begin{remark}\label{rem:boundary-pole}
The boundary-continuity assumption in
Theorem~\ref{thm:general-radius} may be relaxed at isolated boundary
points at which the inverse composition tends to infinity. More
precisely, let
$F=\varphi^{-1}\circ\psi$
be analytic and univalent in $\mathbb D$, and suppose that $F$ extends
continuously to
$\overline{\mathbb D}\setminus E,$
where $E$ is a finite subset of $\partial\mathbb D$, with
$
|F(z)|\to\infty
\;\text{as }z\to\zeta,
\; \zeta\in E.
$
Then the points of $E$ correspond to the point at infinity and hence
do not contribute to the finite distance from the origin to the
boundary of $F(\mathbb D)$. Consequently,
\[
\operatorname{dist}
\bigl(0,\partial F(\mathbb D)\bigr)
=
\min_{\substack{|\zeta|=1\\ \zeta\notin E}}
|F(\zeta)|.
\]
Therefore, in the nontrivial case $\varphi\not\prec\psi$,
\[
\mathcal R(\varphi,\psi)
=
\min_{\substack{|\zeta|=1\\ \zeta\notin E}}
\left|
\varphi^{-1}(\psi(\zeta))
\right|.
\]
\end{remark}
The same radius formula also yields the corresponding radii for the
Ma--Minda starlike and convex classes.

\begin{corollary}\label{cor:maminda-class-radius}
Let $\varphi$ and $\psi$ be Ma--Minda functions satisfying the
hypotheses of Theorem~\ref{thm:general-radius}, and let
$f_r(z)=f(rz)/r$. Then the following assertions hold:
\begin{enumerate}
\item[(i)] if $f\in\mathcal{ST}(\varphi)$, then
$f_r\in\mathcal{ST}(\psi)$ for
$0<r\le\mathcal R(\varphi,\psi)$;

\item[(ii)] if $f\in\mathcal{CV}(\varphi)$, then
$f_r\in\mathcal{CV}(\psi)$ for
$0<r\le\mathcal R(\varphi,\psi)$.
\end{enumerate}
Both radii are sharp. Consequently, the
$\mathcal{ST}(\psi)$-radius of $\mathcal{ST}(\varphi)$ and the
$\mathcal{CV}(\psi)$-radius of $\mathcal{CV}(\varphi)$ are both
equal to $\mathcal R(\varphi,\psi)$.
\end{corollary}
\begin{proof}
For $f_r(z)=f(rz)/r$,
\[
\frac{zf_r'(z)}{f_r(z)}
=
\frac{rzf'(rz)}{f(rz)},
\qquad
1+\frac{zf_r''(z)}{f_r'(z)}
=
1+\frac{rzf''(rz)}{f'(rz)}.
\]
Hence both assertions follow directly from
Theorem~\ref{thm:general-radius}.

For sharpness, let $F_{\varphi},G_{\varphi}\in\mathcal A$ satisfy
\[
\frac{zF_{\varphi}'(z)}{F_{\varphi}(z)}=\varphi(z),
\qquad
1+\frac{zG_{\varphi}''(z)}{G_{\varphi}'(z)}=\varphi(z).
\]
A larger admissible radius for either class would imply
$\varphi(rz)\prec\psi(z)$ for some
$r>\mathcal R(\varphi,\psi)$, contradicting
Theorem~\ref{thm:general-radius}.
\end{proof}

We next verify the inverse-extension and boundary-continuity
conditions for the three functions considered in the subsequent
sections.

\begin{lemma}\label{lem:inverse-hypotheses}
Let $\psi$ be a Ma--Minda function that extends continuously
to $\overline{\mathbb D}$. Then the inverse-extension and
boundary-continuity hypotheses of
Theorem~\ref{thm:general-radius} hold for
$\varphi=\varphi_{\mathrm L}$ and
$\varphi=\varphi_{\mathrm{Lim}}$. They also hold for
$\varphi=\varphi_{\leftmoon}$, provided that $\psi$ has no
zeros on $\partial\mathbb D$.
\end{lemma}

\begin{proof}
Let
$
\mathbb H=\{w\in\mathbb C:\operatorname{Re}w>0\}.
$
Since $\varphi_{\mathrm L}$, $\varphi_{\leftmoon}$,
$\varphi_{\mathrm{Lim}}$, and $\psi$ are Ma--Minda functions, their
images lie in $\mathbb H$. The normalized inverses are
\[
H_{\mathrm L}(w)=w^2-1,\qquad
H_{\leftmoon}(w)=\frac{w^2-1}{2w},\qquad
H_{\mathrm{Lim}}(w)=\sqrt2(\sqrt w-1),
\]
where $\sqrt w$ is the principal square root. These functions are
analytic on $\mathbb H$ and satisfy $H(1)=0$.

The function $H_{\mathrm L}$ is univalent on $\mathbb H$, since
$H_{\mathrm L}(w_1)=H_{\mathrm L}(w_2)$ implies
$(w_1-w_2)(w_1+w_2)=0$, and $w_1+w_2\ne0$ in $\mathbb H$.
Similarly,
$H_{\leftmoon}(w_1)=H_{\leftmoon}(w_2)$ gives
$
(w_1-w_2)(w_1w_2+1)=0.
$
The alternative $w_2=-1/w_1$ is impossible, since
$\operatorname{Re}(-1/w_1)<0$ for $w_1\in\mathbb H$.
Thus $H_{\leftmoon}$ is univalent on $\mathbb H$.
The univalence of $H_{\mathrm{Lim}}$ follows from the injectivity of
the principal square root on $\mathbb H$.

Since $\psi(\overline{\mathbb D})\subset\overline{\mathbb H}$,
the functions $H_{\mathrm L}\circ\psi$ and
$H_{\mathrm{Lim}}\circ\psi$ extend continuously to
$\overline{\mathbb D}$. Moreover, if $\psi$ has no zeros on
$\partial\mathbb D$, then
$H_{\leftmoon}\circ\psi$ also extends continuously to
$\overline{\mathbb D}$. Hence all the required hypotheses are
satisfied.
\end{proof}

\begin{remark}
Each function $\psi$ considered in Sections~3--5 maps $\mathbb D$ onto a Jordan
domain, that is, $\psi(\partial\mathbb D)$ is a simple closed curve (possibly with
cusps). By Carath\'eodory's theorem \cite[Theorem~2.6]{MR1217706}, $\psi$
then extends to a homeomorphism of $\overline{\mathbb D}$ onto
$\overline{\psi(\mathbb D)}$, which gives the continuity of $\psi$ on
$\overline{\mathbb D}$ assumed in Lemma~\ref{lem:inverse-hypotheses}. The
continuity of $\varphi^{-1}\circ\psi$ on $\overline{\mathbb D}$ then follows from
the explicit inverses in the proof of that lemma.
\end{remark}

Combining Theorem~\ref{thm:general-radius} with
Lemma~\ref{lem:inverse-hypotheses} gives the following explicit
formulas for the three functions considered in the subsequent
sections.

\begin{corollary}\label{cor:three-radius-formulas}
Let $\psi$ be a Ma--Minda function that extends continuously
to $\overline{\mathbb D}$. The following formulas hold under
the indicated assumptions.

\begin{enumerate}
\renewcommand{\labelenumi}{(\roman{enumi})}

\item Let $\varphi_{\mathrm L}(z)=\sqrt{1+z}$ and suppose that
$\varphi_{\mathrm L}\not\prec\psi$. Then
\begin{equation}\label{R1-formula}
\mathcal R(\varphi_{\mathrm L},\psi)
=
\min_{|z|=1}|\psi(z)^2-1|
=
\min_{-\pi\le\theta\le\pi}
|\psi(e^{i\theta})^2-1|.
\end{equation}

\item Let
$\varphi_{\leftmoon}(z)=z+\sqrt{1+z^2}$ and suppose that
$\varphi_{\leftmoon}\not\prec\psi$ and
$\psi(z)\ne0$ on $\partial\mathbb D$. Then
\begin{equation}\label{R2-formula}
\mathcal R(\varphi_{\leftmoon},\psi)
=
\min_{|z|=1}
\left|
\frac{\psi(z)^2-1}{2\psi(z)}
\right|
=
\min_{-\pi\le\theta\le\pi}
\left|
\frac{\psi(e^{i\theta})^2-1}
{2\psi(e^{i\theta})}
\right|.
\end{equation}

\item Let
$\varphi_{\mathrm{Lim}}(z)=(1+z/\sqrt2)^2$ and suppose that
$\varphi_{\mathrm{Lim}}\not\prec\psi$. Then
\begin{equation}\label{R3-formula}
\mathcal R(\varphi_{\mathrm{Lim}},\psi)
=
\sqrt2\min_{|z|=1}
\left|\sqrt{\psi(z)}-1\right|
=
\sqrt2\min_{-\pi\le\theta\le\pi}
\left|\sqrt{\psi(e^{i\theta})}-1\right|.
\end{equation}

\end{enumerate}
In \eqref{R3-formula}, the square-root branch is chosen so that
$\sqrt{\psi(0)}=1$.
\end{corollary}

In the applications below, the non-subordination condition can
be verified from the boundary minimum itself. Indeed, under the
inverse-extension and boundary-continuity hypotheses, put
$\Omega=\varphi^{-1}(\psi(\mathbb D))$. If
$\varphi\prec\psi$, then $\mathbb D\subset\Omega$, and hence
\[
\min_{|\zeta|=1}
\left|\varphi^{-1}(\psi(\zeta))\right|
=
\operatorname{dist}(0,\partial\Omega)
\ge1.
\]
Therefore, a boundary minimum strictly less than $1$ implies
$\varphi\not\prec\psi$. The computations in the subsequent
sections verify this inequality for every pair to which
Corollary~\ref{cor:three-radius-formulas} is applied.

For the boundary calculations in the following sections, we use the
standard Chebyshev polynomial identities. Setting $z=e^{i\theta}$
and $x=\cos\theta\in[-1,1]$, we have
\[
z^n+z^{-n}=2T_n(x),
\]
where $T_n$ denotes the $n$th Chebyshev polynomial of the first
kind. In particular,
\[
\begin{aligned}
z+z^{-1}&=2x,\\
z^2+z^{-2}&=4x^2-2,\\
z^3+z^{-3}&=8x^3-6x,\\
z^4+z^{-4}&=16x^4-16x^2+2.
\end{aligned}
\]
These identities will be used to express the boundary-modulus
functions as real polynomials in $x$.
\section{Sharp Radii for the Lemniscate-Associated Class}

In this section, we determine the sharp $\mathcal P_{\psi}$-radii of
$\mathcal P_{\varphi_{\mathrm L}}$ for ten choices of $\psi$.
The corresponding sharp radii for the associated Ma--Minda starlike
and convex classes follow from Corollary~\ref{cor:maminda-class-radius}.

The exponential starlike and convex classes were introduced by
Mendiratta \textit{et al.}~\cite{MR3394060} in 2015 as special cases
of the Ma--Minda classes associated with the function
$\psi_e(z)=e^z$.
The following theorem determines the sharp inclusion radius of the
lemniscate-associated class in the exponential class.

\begin{theorem}
The $\mathcal{P}_{\psi_e}$-radius of the class
$\mathcal{P}_{\varphi_{\mathrm L}}$ is
\[
\mathcal{R}(\varphi_{\mathrm L},\psi_e)
=
1-e^{-2}
\approx 0.864665.
\]
\end{theorem}

\begin{proof}

For $\psi_e(z)=e^z$, it follows from \eqref{R1-formula} that
\[
\mathcal{R}(\varphi_{\mathrm L},\psi_e)
=
\min_{-\pi\le\theta\le\pi}
|\psi_e(e^{i\theta})^2-1|.
=
\min_{-\pi\le\theta\le\pi}
|e^{2e^{i\theta}}-1|.
\]
Consider the squared modulus
\begin{equation}\label{exp-eq1}
F(\theta)
=
|e^{2e^{i\theta}}-1|^2
=
e^{4\cos\theta}
-2e^{2\cos\theta}\cos(2\sin\theta)
+1.
\end{equation}
Since $F$ is even on $[-\pi,\pi]$, it is sufficient to minimize it
on $[0,\pi]$. Differentiating, we obtain
\[
F'(\theta)
=
-4e^{2\cos\theta}
\left[
\sin\theta\,e^{2\cos\theta}
-
\sin(\theta+2\sin\theta)
\right].
\]
For $0<\theta\le\pi/2$, using
$\sin u<u$ for $u>0$ and $1+t\le e^t$, we have
\begin{align*}
\sin(\theta+2\sin\theta)
&=
\sin\theta\cos(2\sin\theta)
+\cos\theta\sin(2\sin\theta)\\
&<
\sin\theta+2\cos\theta\sin\theta\\
&=
\sin\theta(1+2\cos\theta)\\
&\le
\sin\theta\,e^{2\cos\theta}.
\end{align*}
For $\pi/2<\theta<\pi$, set $u=\pi-\theta\in(0,\pi/2)$. By the
concavity of $\sin u$ on $[0,\pi/2]$,
$
\sin u\ge(2u/\pi),
$
and hence $2\sin u>u$. Therefore,
$\theta+2\sin\theta=\pi-u+2\sin u>\pi$.
Also, $\theta+2\sin\theta<\pi+2<2\pi$, and consequently
\[
\sin(\theta+2\sin\theta)<0
<
\sin\theta\,e^{2\cos\theta}.
\]
Thus,
$\sin\theta\,e^{2\cos\theta}
-\sin(\theta+2\sin\theta)>0$
for $0<\theta<\pi$. Hence $F'(\theta)<0$ on $(0,\pi)$, so
$F$ is strictly decreasing on $[0,\pi]$. Therefore, the minimum is
attained at $\theta=\pi$, where
$F(\pi)=e^{-4}-2e^{-2}+1=(1-e^{-2})^2$.
Consequently,
\[
\mathcal{R}(\varphi_{\mathrm L},\psi_e)
=
\sqrt{F(\pi)}
=
1-e^{-2}.
\]
Since the minimum is attained at $z_0=-1$, we have
$\psi_e(z_0)=e^{-1}$. Let $r_0=1-e^{-2}$. Then
$\varphi_{\mathrm L}(-r_0)=\sqrt{1-r_0}
=e^{-1}=\psi_e(-1)$.
Hence the radius is sharp.
\end{proof}

The lune starlike and convex classes were introduced by
Raina and Sokół~\cite{MR3419845} in 2015 as Ma--Minda classes
associated with the function
$
\psi_{\leftmoon}(z)=z+\sqrt{1+z^2}.
$
In particular,
$\mathcal{ST}_{\leftmoon}=\mathcal{ST}(\psi_{\leftmoon})$ and
$\mathcal{CV}_{\leftmoon}=\mathcal{CV}(\psi_{\leftmoon})$.
The image of $\mathbb D$ under $\psi_{\leftmoon}$ is the lune-shaped domain
\[
\left\{
w\in\mathbb C:
\operatorname{Re}w>0,\; |w^2-1|<2|w|
\right\}.
\]
The following theorem determines the sharp inclusion radius of the
lemniscate-associated class in the lune class.
Writing $z=e^{i\theta}$ and $x=\cos\theta$, we divide
the boundary into two portions according to the sign of $x$.
For $x \in [0, 1]$, which corresponds to $\theta \in [-\pi/2, \pi/2]$, the curve traces the major arc of the lune. Its parametric representation is given by:
$$\psi_{{\leftmoon}_{\text{major}}}(x) = \bigl(x+\sqrt{x(1+x)}\bigr) \pm i\bigl(\sqrt{1-x^2}+\sqrt{x(1-x)}\bigr).$$
For $x \in [-1, 0]$, which corresponds to the remaining intervals $\theta \in [-\pi, -\pi/2] \cup [\pi/2, \pi]$, the curve traces the minor arc, represented by:
\[
\psi_{{\leftmoon}_{\text{minor}}}(x) = \bigl(x+\sqrt{-x(1-x)}\bigr) \pm i\bigl(\sqrt{1-x^2}-\sqrt{-x(1+x)}\bigr).
\]

\begin{theorem}
The $\mathcal{P}_{\psi_{\leftmoon}}$-radius of the class $\mathcal{P}_{\varphi_{\mathrm{L}}}$ is
\[
\mathcal{R}(\varphi_{\mathrm{L}},\psi_{\leftmoon})=2(\sqrt2-1) \approx 0.828427.
\]
\end{theorem}

\begin{proof}
Given $\psi_{\leftmoon}(z)=z+\sqrt{1+z^2}$, then by \eqref{R1-formula},
\[
\mathcal{R}(\varphi_{\mathrm{L}},\psi_{\leftmoon})
=\min_{|z|=1}\left| \psi_{\leftmoon}(z) ^{2}-1\right|=\min_{|z|=1}\left|\left( z+\sqrt{1+z^2}\right) ^{2}-1\right|
\]
Since
$
\left(z+\sqrt{1+z^2}\right)^2-1
=2z\,\psi_{\leftmoon}(z),
$ set $z=e^{i\theta}, \theta \in [-\pi, \pi]$
we obtain
\begin{equation}\label{lune-eq1}
\mathcal{R}(\varphi_{\mathrm L},\psi_{\leftmoon})
=2\min_{-\pi\le\theta\le\pi}
\left|\psi_{\leftmoon}(e^{i\theta})\right|.
\end{equation}
For $x\in[0,1]$, the boundary parametrization gives
$
|\psi_{\leftmoon}(e^{i\theta})|^2
=
1+2x+2\sqrt{x^2+x}
\ge1.
$
On the other hand,
$
|\psi_{\leftmoon}(-1)|^2=(\sqrt2-1)^2<1.
$
Therefore, the global minimum occurs on the boundary
portion corresponding to $x\in[-1,0]$.
Using the minor arc 
representation,
\[
\psi_{\leftmoon}(e^{i\theta})
=
\bigl(x+\sqrt{-x(1-x)}\bigr)
\pm
i\bigl(\sqrt{1-x^2}-\sqrt{-x(1+x)}\bigr).
\]
Thus, setting
$
F(x):=|\psi_{\leftmoon}(e^{i\theta})|^2
=
1-2x-2\sqrt{x^2-x},
$
we minimize $F$ on $[-1,0]$. Differentiating, we obtain
\[
F'(x)
=
-2-\frac{2x-1}{\sqrt{x^2-x}},
\qquad -1<x<0.
\]
Setting $F'(x)=0$ gives
$
-2\sqrt{x^2-x}=2x-1
\quad\Longrightarrow\quad
4(x^2-x)=(2x-1)^2,
$
which reduces to the contradiction $0=1$. Hence $F$ has no critical point in
$(-1,0)$, so its minimum on $[-1,0]$ occurs at an endpoint. Since
$
F(0)=1,
F(-1)=3-2\sqrt2=(\sqrt2-1)^2,
$
we obtain, by \eqref{lune-eq1},
\[
\mathcal{R}(\varphi_{\mathrm{L}},\psi_{\leftmoon})
=
2\min_{-\pi \le \theta \le \pi}\left|\psi_{\leftmoon}(e^{i\theta})\right|
=
2(\sqrt2-1).
\]
Let $r_0=2(\sqrt2-1)$. At the boundary point $z_0=-1$,
$\psi_{\leftmoon}(z_0)=\sqrt2-1,$ while $\varphi_{\mathrm L}(-r_0)=\sqrt{1-r_0}=\sqrt2-1.$
Hence
$\varphi_{\mathrm L}(-r_0)=\psi_{\leftmoon}(-1),$
and therefore the radius is sharp.
\end{proof}

In 2023, Mundalia and Kumar~\cite{MR4620985} introduced the
Ma--Minda classes associated with the hyperbolic cosine function
$
\psi_{\varrho}(z)=\cosh\sqrt{z}.
$
The corresponding starlike and convex classes are denoted by
$\mathcal{ST}_{\varrho}=\mathcal{ST}(\psi_{\varrho})$ and
$\mathcal{CV}_{\varrho}=\mathcal{CV}(\psi_{\varrho})$, respectively.
The following theorem determines the sharp inclusion radius of the
lemniscate-associated class in the hyperbolic cosine class.

\begin{theorem}
The $\mathcal{P}_{\psi_{\varrho}}$-radius of the class $\mathcal{P}_{\varphi_{\mathrm{L}}}$ is
\[
\mathcal{R}(\varphi_{\mathrm{L}},\psi_{\varrho})=\sin^2(1) \approx 0.708073.
\]
\end{theorem}

\begin{proof}
Given $\psi_{\varrho}(z)=\cosh(\sqrt{z})$, it follows from \eqref{R1-formula} that the radius is
\[
\mathcal{R}(\varphi_{\mathrm{L}},\psi_{\varrho})
=
\min_{|z|=1}\left|\psi_{\varrho}(z)^2-1\right|
=
\min_{|z|=1}\left|\cosh^2(\sqrt{z})-1\right|
=
\min_{|z|=1}\left|\sinh^2(\sqrt{z})\right|.
\]
Set $z=e^{i\theta}$, where $-\pi\le\theta\le\pi$, and let
$w=\sqrt{z}=u+iv$. For the principal square root,
\[
w=e^{i\theta/2},
\qquad
u^2+v^2=1,
\qquad
0\le u\le1.
\]
We have
$
\left|\sinh(u+iv)\right|^2
=
\sinh^2(u)\cos^2(v)
+
\cosh^2(u)\sin^2(v).
$
Using the identity $\cosh^2(u)=1+\sinh^2(u)$, this expression simplifies to
\[
\left|\sinh(w)\right|^2
=
\sinh^2(u)(\cos^2(v)+\sin^2(v))
+\sin^2(v)
=
\sinh^2(u)+\sin^2(v).
\]
Since $v^2=1-u^2$, we obtain
$|\sinh(w)|^2=\sinh^2(u)+\sin^2\!\left(\sqrt{1-u^2}\right)$, and hence define
$F(u)=\sinh^2(u)+\sin^2\!\left(\sqrt{1-u^2}\right)$ for $u\in[0,1]$.
Differentiating with respect to $u$, we obtain
\[
F'(u)
=
\sinh(2u)
-
u\frac{\sin\left(2\sqrt{1-u^2}\right)}
{\sqrt{1-u^2}},
\qquad
u\in(0,1).
\]
For $u\in(0,1)$, we have $\sinh(2u)>2u$. Also, since
$2\sqrt{1-u^2}>0$ and $\sin x<x$ for $x>0$,
$
\sin\left(2\sqrt{1-u^2}\right)
<2\sqrt{1-u^2}.
$
Therefore,
\[
u\frac{\sin\left(2\sqrt{1-u^2}\right)}
{\sqrt{1-u^2}}
<
2u
<
\sinh(2u),
\]
which gives
\[
F'(u)
=
\sinh(2u)
-
u\frac{\sin\left(2\sqrt{1-u^2}\right)}
{\sqrt{1-u^2}}
>
0.
\]
Thus $F$ is strictly increasing on $[0,1]$, and hence its minimum occurs at $u=0$. Since $u=\cos(\theta/2)$, this corresponds to $\theta=\pm\pi$ and hence $z=-1$. Therefore,
\[
\mathcal{R}(\varphi_{\mathrm{L}},\psi_{\varrho})
=F(0)=\sin^2(1).
\]

Since the minimum is attained at $z_0=-1$, we have
$\psi_{\varrho}(z_0)=\cos(1)$.
Let $r_0=\sin^2(1)$. Then
$\varphi_{\mathrm L}(-r_0)=\sqrt{1-r_0}=\cos(1)=\psi_{\varrho}(-1)$.
Hence the radius $r_0$ is sharp.
\end{proof}

The cardioid starlike and convex classes were introduced by
Sharma \textit{et al.}~\cite{MR3536076} as special cases of the
Ma--Minda classes associated with the function
\[
\psi_C(z)=1+\frac{4}{3}z+\frac{2}{3}z^2.
\]
In particular, these classes are given by
$\mathcal{ST}_C=\mathcal{ST}(\psi_C)$ and
$\mathcal{CV}_C=\mathcal{CV}(\psi_C)$.
The following theorem determines the sharp inclusion radius of the
lemniscate-associated class in the cardioid class.
The following radius is from
\cite[Theorem~4.9]{MR3536076}.
We recover it through an analytical boundary minimization.

\begin{theorem}
The $\mathcal{P}_{\psi_C}$-radius of the class $\mathcal{P}_{\varphi_{\mathrm{L}}}$ is
\[
\mathcal{R}(\varphi_{\mathrm{L}},\psi_C)=\frac89\approx 0.8889.
\]
\end{theorem}

\begin{proof}
For $\psi_C(z)=1+4z/3+2z^2/3$, using
\eqref{R1-formula}, we obtain
\[
\mathcal{R}(\varphi_{\mathrm L},\psi_C)
=\min_{|z|=1}|\psi_C(z)^2-1|.
\]
Factoring, we have
$
\psi_C(z)^2-1=\left( 4/9\right) z(z+2)(z^2+2z+3),
$
and hence
\begin{equation}\label{car-eq1}
|\psi_C(z)^2-1|^2
=\frac{16}{81}|z|^2|z+2|^2|z^2+2z+3|^2.
\end{equation}
Setting $z=e^{i\theta}$ and $x=\cos\theta\in[-1,1]$, the
Chebyshev polynomial identities give
$
|z+2|^2=4x+5,\;
|z^2+2z+3|^2=4(3x^2+4x+2).
$
Thus, \eqref{car-eq1} becomes
\begin{equation}\label{car-eq2}
|\psi_C(e^{i\theta})^2-1|^2
=\frac{64}{81}\left( 12x^3+31x^2+28x+10\right) 
\end{equation}
Let $F(x) = 12x^3 + 31x^2 + 28x + 10$, where $x \in [-1,1]$. Differentiating with respect to $x$, we get
$
F'(x) = 36x^2 + 62x + 28.
$
Since  $F'(x) $ is a quadratic with positive leading coefficient and negative discriminant $(\Delta=-188)$, so that $F'(x)>0$ for all $x\in \mathbb{R}$, thus $F(x)$ is strictly increasing on the interval $[-1, 1]$. Therefore, its minimum is $1$ attained precisely at the left endpoint, $x = -1$.
\noindent Consequently,
\[
\mathcal{R}(\varphi_{\mathrm{L}},\psi_C)
=
\min_{|z|=1}\left|\psi_C(z)^2-1\right|
=
\sqrt{\frac{64}{81}\times F(-1)}
=
\sqrt{\frac{64}{81}}=\frac{8}{9}.
\]
Since the minimum is attained at $z_0=-1$, we have
$\psi_C(z_0)=1/3.$
Let $r_0=8/9$. Then 
$\varphi_{\mathrm L}(-r_0)=\sqrt{1-r_0}=1/3=\psi_C(-1)$.
Hence the radius is sharp.
\end{proof}

 In 2021, Gupta \textit{et al.}~\cite{MR4303808} introduced the
 Ma--Minda starlike and convex classes associated with the cardioid
 function
 \[
 \psi_{\mathrm{car}}(z)
 =
 1+z+\frac{z^2}{2}.
 \]
 In particular,
 $\mathcal{ST}_{\mathrm{car}}=\mathcal{ST}(\psi_{\mathrm{car}})$ and
 $\mathcal{CV}_{\mathrm{car}}=\mathcal{CV}(\psi_{\mathrm{car}})$.
 The image of $\mathbb D$ under $\psi_{\mathrm{car}}$ is the domain
 bounded by the cardioid
 \[
 \left(4x^2+4y^2-8x-1\right)^2
 +
 4\left(4x^2+4y^2-12x+1\right)
 =
 0.
 \]
 The following theorem determines the sharp inclusion radius of the
 lemniscate-associated class in the cardioid class.
 
 The following radius is the case $c=1$ of
 \cite[Theorem~4.3(a)]{MR4303808}.
 We recover it through an analytical boundary minimization.

 \begin{theorem} 
 The $\mathcal{P}_{\psi_{car}}$-radius of the class
 $\mathcal{P}_{\varphi_{\mathrm L}}$ is 
 \[
 \mathcal{R}(\varphi_{\mathrm L},\psi_{car})
 =
 \frac34
 =
 0.75.
 \]
 \end{theorem} 
  
 \begin{proof} 
 
 For $\psi_{car}(z)=1+z+z^2/2$, it follows from
 \eqref{R1-formula} that
 \[
 \mathcal{R}(\varphi_{\mathrm L},\psi_{car})
 =
 \min_{|z|=1}|\psi_{car}(z)^2-1|
 \]
 The expression $\psi_{car}(z)^2-1$ can be factored as
 $ (z/4)(2+z)(4+2z+z^2),$
and hence
 \begin{equation}\label{Acar-eq1} 
 |\psi_{car}(z)^2-1|^2
 =
 \frac{1}{16}|z|^2|2+z|^2|4+2z+z^2|^2.
 \end{equation}
 Set $z=e^{i\theta}$ and $x=\cos\theta\in[-1,1]$. By
 Chebyshev polynomial identities, we obtain
 $|2+z|^2=4x+5$ and
 $|4+2z+z^2|^2=16x^2+20x+13$.
 Then \eqref{Acar-eq1}, simplifies to
 \begin{equation}\label{Acar-eq2} 
 |\psi_{car}(e^{i\theta})^2-1|^2
 =
 \frac{1}{16}
 \left(
 64x^3+160x^2+152x+65
 \right).
 \end{equation}
 Let $F(x)=64x^3+160x^2+152x+65$ for $x\in[-1,1]$. Differentiating,
 we obtain
 \[
 F'(x)
 =
 192x^2+320x+152
 =
 8(24x^2+40x+19).
 \]
 The quadratic $24x^2+40x+19$ has negative discriminant
 $\Delta=-224<0$ and positive leading coefficient. Hence
 $F'(x)>0$ for all $x\in\mathbb R$, so $F$ is strictly increasing on
 $[-1,1]$. Therefore, its minimum is attained at $x=-1$, where
 $F(-1)=9$.
 
 Consequently,
 \[
 \mathcal{R}(\varphi_{\mathrm L},\psi_{car})
 =
 \sqrt{\frac{1}{16}F(-1)}
 =
 \sqrt{\frac{9}{16}}
 =
 \frac34.
 \]
 Since the minimum is attained at $z_0=-1$, we have
 $\psi_{car}(z_0)=1/2$. Let $r_0=3/4$. Then
 $\varphi_{\mathrm L}(-r_0)=\sqrt{1-r_0}=1/2=\psi_{car}(-1)$.
 Hence the radius is sharp.
 \end{proof}
 
 In 2021, Wani and Swaminathan~\cite{MR4190740} introduced the
 Ma--Minda starlike and convex classes associated with the nephroid
 function
 \[
 \psi_{\mathrm{Ne}}(z)
 =
 1+z-\frac{z^3}{3}.
 \]
 In particular,
 $\mathcal{ST}_{\mathrm{Ne}}=\mathcal{ST}(\psi_{\mathrm{Ne}})$ and
 $\mathcal{CV}_{\mathrm{Ne}}=\mathcal{CV}(\psi_{\mathrm{Ne}})$.
 The image of $\mathbb D$ under $\psi_{\mathrm{Ne}}$ is the interior of
 the nephroid bounded by the two-cusped curve
 \[
 \left((u-1)^2+v^2-\frac{4}{9}\right)^3
 -\frac{4v^2}{3}
 =
 0.
 \]
 The following theorem determines the sharp inclusion radius of the
 lemniscate-associated class in the nephroid class.
 The following radius was obtained in
 \cite[Theorem~3.2(ii)]{MR4129627}
 for $\alpha=0$.
 We provide a proof using the general subordination radius formula.
 \begin{theorem}
 The $\mathcal{P}_{\psi_{\mathrm{Ne}}}$-radius of the class
 $\mathcal{P}_{\varphi_{\mathrm L}}$ is
 \[
 \mathcal{R}(\varphi_{\mathrm L},\psi_{\mathrm{Ne}})
 =
 \frac89
 \approx 0.8889.
 \]
 \end{theorem}
 
 \begin{proof}
 For $\psi_{\mathrm{Ne}}(z)=1+z-z^3/3$, it follows from
 \eqref{R1-formula} that
 \[
 \mathcal{R}(\varphi_{\mathrm L},\psi_{\mathrm{Ne}})
 =
 \min_{|z|=1}|\psi_{\mathrm{Ne}}(z)^2-1|
 \]
Factoring $\psi_{\mathrm{Ne}}(z)^2-1$ and taking the modulus squared, we obtain
\begin{equation}\label{nep-eq1}
|\psi_{\mathrm{Ne}}(z)^2-1|^2
=\frac1{81}|z|^2|3-z^2|^2|6+3z-z^3|^2.
\end{equation}
 Set $z=e^{i\theta}$ and $x=\cos\theta\in[-1,1]$. Using
 Chebyshev polynomial identities we obtain
 $|3-z^2|^2=4(4-3x^2)$ and
 $|6+3z-z^3|^2=4(13+18x-3x^2-12x^3)$.
Then \eqref{nep-eq1} becomes
 \begin{equation}\label{nep-eq2}
 |\psi_{\mathrm{Ne}}(e^{i\theta})^2-1|^2
 =
 \frac{16}{81}
 (4-3x^2)(13+18x-3x^2-12x^3).
 \end{equation}
 Let
 $F(x)=36x^5+9x^4-102x^3-51x^2+72x+52$.
 To determine its minimum on $[-1,1]$, we show that
 $F(x)\ge4$. Expanding and subtracting $4$, we obtain
 \begin{align*}
 F(x)-4
 &= 36x^5+9x^4-102x^3-51x^2+72x+48\\
& = 3(x+1)(12x^4-9x^3-25x^2+8x+16).
 \end{align*}
 Since $3(x+1)\ge0$ on $[-1,1]$, it remains to prove that
 $R(x)=12x^4-9x^3-25x^2+8x+16>0$.
 Write $R(x)=A(x)+B(x)$, where
 $A(x)=12x^4-25x^2+13$ and
 $B(x)=-9x^3+8x+3$. Since
 $
 A(x)
 = (x^2-1)(12x^2-13),
 $
 and both factors are nonpositive on $[-1,1]$, we have
 $A(x)\ge0$.
 For $B(x)=-9x^3+8x+3$, we have
 $B'(x)=8-27x^2$, so the critical points are
 $x=\pm\sqrt{8/27}$. Comparing the values at the critical points and
 the endpoints gives
 \[
 \min_{x\in[-1,1]}B(x)
 =
 B\left(-\sqrt{\frac{8}{27}}\right)
 =
 3-\frac{32\sqrt6}{27}
 >
 0,
 \]
 where the last inequality follows from $81^2>32^2\cdot6$.
 Thus $R(x)=A(x)+B(x)>0$ on $[-1,1]$.
 
\noindent Consequently,
 $F(x)-4=3(x+1)R(x)\ge0$ on $[-1,1]$, with equality only at
 $x=-1$. Hence the minimum of $F$ is $F(-1)=4$. Therefore,
 \[
 \mathcal{R}(\varphi_{\mathrm L},\psi_{\mathrm{Ne}})
 =
 \sqrt{\frac{16}{81}F(-1)}
 =
 \sqrt{\frac{64}{81}}
 =
 \frac89.
 \]
 
 Since the minimum is attained at $z_0=-1$, we have
 $\psi_{\mathrm{Ne}}(z_0)=1/3$. Let $r_0=8/9$. Then
 $\varphi_{\mathrm L}(-r_0)=\sqrt{1-r_0}=1/3
 =\psi_{\mathrm{Ne}}(-1)$.
 Hence the radius is sharp.
 \end{proof}
 
 Gandhi~\cite{MR4099174} introduced the Ma--Minda starlike class
 associated with the three-leaf function
 \[
 \psi_{3L}(z)
 =
 1+\frac{4}{5}z+\frac{1}{5}z^4.
 \]
 The corresponding starlike and convex classes are denoted by
 $\mathcal{ST}_{3L}=\mathcal{ST}(\psi_{3L})$ and
 $\mathcal{CV}_{3L}=\mathcal{CV}(\psi_{3L})$, respectively.
 The following theorem determines the sharp inclusion radius of the
 lemniscate-associated class in the three-leaf class.
 
 The following radius is the case $n=4$ and $c=1$
 of \cite[Theorem~5.2(b)]{MR4510346}.
 We recover it through the general subordination radius
 formula and an explicit polynomial positivity argument.
 
 \begin{theorem}
 The $\mathcal{P}_{\psi_{3L}}$-radius of the class
 $\mathcal{P}_{\varphi_{\mathrm L}}$ is
 \[
 \mathcal{R}(\varphi_{\mathrm L},\psi_{3L})
 =
 \frac{21}{25}
 =
 0.84.
 \]
 \end{theorem}
 
 \begin{proof}
 
 For $\psi_{3L}(z)=1+4z/5+z^4/5$, it follows from
 \eqref{R1-formula} that
 \[
 \mathcal{R}(\varphi_{\mathrm L},\psi_{3L})
 =
 \min_{|z|=1}|\psi_{3L}(z)^2-1|.
 \]
 Factoring $\psi_{3L}(z)^2-1$ and taking the squared modulus, we obtain
 \begin{equation}\label{3L-eq1}
 |\psi_{3L}(z)^2-1|^2
 =
 \frac{1}{625}|z|^2|4+z^3|^2|10+4z+z^4|^2.
 \end{equation}
 Set $z=e^{i\theta}$ and $x=\cos\theta\in[-1,1]$. Using
 Chebyshev polynomial identities we obtain
 $|4+z^3|^2=32x^3-24x+17$ and
 $|10+4z+z^4|^2
 =160x^4+32x^3-160x^2+56x+137$.
 Substituting these expressions into \eqref{3L-eq1}, we obtain
 \begin{equation}\label{3L-eq2}
 |\psi_{3L}(e^{i\theta})^2-1|^2
 =
 \frac{1}{625}
 (32x^3-24x+17)
 (160x^4+32x^3-160x^2+56x+137).
 \end{equation}
 Let
 $F(x)=(32x^3-24x+17)
 (160x^4+32x^3-160x^2+56x+137)$.
 To show that $F(x)\ge441$ on $[-1,1]$, we consider
 \[
 F(x)-441
 =
 5120x^7+1024x^6-8960x^5+3744x^4
 +8768x^3-4064x^2-2336x+1888.
 \]
 This can be factored as
 $
 F(x)-441=16(x+1)A(x),
 $
 where
 $A(x)=320x^6-256x^5-304x^4+538x^3+10x^2-264x+118$.
 Since $16(x+1)\ge0$ on $[-1,1]$, it remains to show that
 $A(x)>0$.
 For this purpose, we write
 $
 2A(x)= (2x-1)^2B(x)+45(2-x),
 $
 where
 $B(x)=160x^4+32x^3-160x^2+101x+146$.
 Moreover,
 \[
 B(x)
 =
 10\left(4x^2+\frac25x-3\right)^2
 +
 \frac15(392x^2+625x+280).
 \]
 The first term is nonnegative. The quadratic
 $392x^2+625x+280$ has discriminant
 $\Delta=-48415<0$ and positive leading coefficient, so it is
 strictly positive for all $x\in\mathbb R$. Hence $B(x)>0$ on
 $[-1,1]$. Since $2-x>0$ on this interval, it follows that
 $A(x)>0$.
 Therefore,
 $
 F(x)-441=16(x+1)A(x)\ge0,
 $
 with equality only at $x=-1$. Hence the minimum of $F$ on $[-1,1]$
 is $F(-1)=441$. Thus,
 \[
 \mathcal{R}(\varphi_{\mathrm L},\psi_{3L})
 =
 \sqrt{\frac{1}{625}F(-1)}
 =
 \sqrt{\frac{441}{625}}
 =
 \frac{21}{25}.
 \]
 
 Since the minimum is attained at $z_0=-1$, we have
 $\psi_{3L}(z_0)=2/5$. Let $r_0=21/25$. Then
 $\varphi_{\mathrm L}(-r_0)=\sqrt{1-r_0}=2/5=\psi_{3L}(-1)$.
 Hence the radius is sharp.
 \end{proof}
 
 Masih \textit{et al.}~\cite{Masih} studied the Ma--Minda starlike and
 convex classes associated with the family of functions
 \[
 \psi_s(z)=(1+sz)^2,
 \qquad
 0<s\le\frac{1}{\sqrt2}.
 \]
 In the present work, we consider the extremal case $s=1/\sqrt2$, which
 gives the lima\c{c}on function
 \[
 \psi_{\mathrm{Lim}}(z)
 =
 \left(1+\frac{z}{\sqrt2}\right)^2.
 \]
 In particular,
 $\mathcal{ST}_{\mathrm{Lim}}=\mathcal{ST}(\psi_{\mathrm{Lim}})$ and
 $\mathcal{CV}_{\mathrm{Lim}}=\mathcal{CV}(\psi_{\mathrm{Lim}})$.
 Related inclusion properties for this class were studied by Kanaga and
 Ravichandran~\cite{MR4478088}; see also~\cite{MR4386722}.
 The following theorem determines the sharp inclusion radius of the
 lemniscate-associated class in the lima\c{c}on class.
 
 The following radius is the case $\alpha=0$ of
 \cite[Theorem~2.6(ii)]{MR4386722}.
 We recover it by evaluating the corresponding boundary minimum.
 \begin{theorem}
 The $\mathcal{P}_{\psi_{\mathrm{Lim}}}$-radius of the class
 $\mathcal{P}_{\varphi_{\mathrm L}}$ is
 \[
 \mathcal{R}(\varphi_{\mathrm L},\psi_{\mathrm{Lim}})
 =
 3\sqrt2-\frac{13}{4}
 \approx 0.99264.
 \]
 \end{theorem}
 
 \begin{proof}
 
 For $\psi_{\mathrm{Lim}}(z)=1+\sqrt2 z+z^2/2$, it follows from
 \eqref{R1-formula} that
 \[
 \mathcal{R}(\varphi_{\mathrm L},\psi_{\mathrm{Lim}})
 =
 \min_{|z|=1}|\psi_{\mathrm{Lim}}(z)^2-1|
 =
 \min_{-\pi\le\theta\le\pi}
 |\psi_{\mathrm{Lim}}(e^{i\theta})^2-1|.
 \]
 Using the factorization of $\psi_{\mathrm{Lim}}(z)^2-1$, we obtain
 \[
 \psi_{\mathrm{Lim}}(z)^2-1
 =
 \frac{1}{4}z(2\sqrt2+z)(4+2\sqrt2 z+z^2).
 \]
 Hence
 \begin{equation}\label{lim-eq1}
 |\psi_{\mathrm{Lim}}(z)^2-1|^2
 =
 \frac{1}{16}|z|^2
 |2\sqrt2+z|^2
 |4+2\sqrt2 z+z^2|^2.
 \end{equation}
 Setting $z=e^{i\theta}$ and $x=\cos\theta\in[-1,1]$, a direct
 calculation gives
 $|2\sqrt2+z|^2=4\sqrt2 x+9$ and
 $|4+2\sqrt2 z+z^2|^2=16x^2+20\sqrt2 x+17$.
 Substituting these expressions into \eqref{lim-eq1}, we obtain
 \begin{equation}\label{lim-eq2}
 |\psi_{\mathrm{Lim}}(e^{i\theta})^2-1|^2
 =
 4\sqrt2 x^3+19x^2
 +\frac{31\sqrt2}{2}x+\frac{153}{16}.
 \end{equation}
 Let
 \[
 F(x)
 =
 4\sqrt2 x^3+19x^2
 +\frac{31\sqrt2}{2}x+\frac{153}{16}.
 \]
 Then
 \[
 F'(x)
 =
 12\sqrt2 x^2+38x+\frac{31\sqrt2}{2}.
 \]
 The discriminant of $F'(x)$ is $\Delta=-44<0$, and its leading
 coefficient is positive. Hence $F'(x)>0$ for all $x\in\mathbb R$.
 Therefore, $F$ is strictly increasing on $[-1,1]$ and attains its
 minimum at $x=-1$, where
 \[
 F(-1)
 =
 \frac{457-312\sqrt2}{16}
 =
 \frac{(12\sqrt2-13)^2}{16}.
 \]
 Consequently,
 \[
 \mathcal{R}(\varphi_{\mathrm L},\psi_{\mathrm{Lim}})
 =
 \sqrt{F(-1)}
 =
 3\sqrt2-\frac{13}{4}.
 \]
 
 Since the minimum is attained at $z_0=-1$, we have
 $\psi_{\mathrm{Lim}}(z_0)=3/2-\sqrt2$.
 Let $r_0=3\sqrt2-13/4$. Then
 $\varphi_{\mathrm L}(-r_0)=\sqrt{1-r_0}
 =3/2-\sqrt2=\psi_{\mathrm{Lim}}(-1)$.
 Hence the radius is sharp.
 \end{proof}

The Ma--Minda starlike and convex classes associated with the rational function were introduced by Kumar and Ravichandran~\cite{MR3496681}. These classes correspond to the function
\[
\psi_R(z)
=
1+\frac{z^2+kz}{k^2-kz},
\qquad
k=\sqrt2+1.
\]
In particular,
$\mathcal{ST}_{R}=\mathcal{ST}(\psi_R)$ and
$\mathcal{CV}_{R}=\mathcal{CV}(\psi_R)$.
The following theorem determines the sharp inclusion radius of the
lemniscate-associated class in the rational-function class.
The following radius is from
\cite{MR3496681}.
We recover it through an analytical boundary minimization.

\begin{theorem}
The $\mathcal{P}_{\psi_R}$-radius of the class
$\mathcal{P}_{\varphi_{\mathrm L}}$ is
\[
\mathcal{R}(\varphi_{\mathrm L},\psi_R)
=
8\sqrt2-11
\approx 0.313708.
\]
\end{theorem}

\begin{proof}

Let
$\psi_R(z)=1+(z^2+kz)/(k^2-kz)$, where $k=\sqrt2+1$.
It follows from \eqref{R1-formula} that
\[
\mathcal{R}(\varphi_{\mathrm L},\psi_R)
=
\min_{|z|=1}|\psi_R(z)^2-1|.
\]
Factoring the expression inside the modulus, we obtain
\begin{equation}\label{Rational-eq1}
\psi_R(z)^2-1
=
\frac{z(k+z)(z^2-kz+2k^2)}
{(k^2-kz)^2}.
\end{equation}
Hence, for $|z|=1$,
\[
|\psi_R(z)^2-1|^2
=
\frac{|k+z|^2|z^2-kz+2k^2|^2}
{|k^2-kz|^4}.
\]
Setting $z=e^{i\theta}$ and $x=\cos\theta\in[-1,1]$, a direct
calculation gives
$|k+z|^2=k^2+2kx+1$, $
|z^2-kz+2k^2|^2
=1-3k^2+4k^4-(2k+4k^3)x+8k^2x^2,
$
and
$|k^2-kz|^4=k^4(k^2-2kx+1)^2$.
Therefore,
\begin{equation}\label{Rational-eq2}
|\psi_R(e^{i\theta})^2-1|^2
=
\frac{(k^2+2kx+1)
\left(1-3k^2+4k^4-(2k+4k^3)x+8k^2x^2\right)}
{k^4(k^2-2kx+1)^2}.
\end{equation}
Substituting $k=\sqrt2+1$ in \eqref{Rational-eq2}, we obtain
\begin{equation}\label{Rational-eq3}
|\psi_R(e^{i\theta})^2-1|^2
=
\frac{(x+\sqrt2)
\left(4x^2-(1+3\sqrt2)x+6+3\sqrt2\right)}
{(1+\sqrt2)^3(x-\sqrt2)^2}.
\end{equation}
Let
\[
F(x)
=
\frac{(x+\sqrt2)
\left(4x^2-(1+3\sqrt2)x+6+3\sqrt2\right)}
{(1+\sqrt2)^3(x-\sqrt2)^2}.
\]
Then
\[
F'(x)
=
\frac{4\left(x^3-3\sqrt2 x^2-x-(4+3\sqrt2)\right)}
{(1+\sqrt2)^3(x-\sqrt2)^3}.
\]
For $x\in[-1,1]$, we have $x^3\le1$, $-x\le1$, and
$-3\sqrt2 x^2\le0$. Hence
\[
x^3-3\sqrt2 x^2-x-(4+3\sqrt2)
\le
-2-3\sqrt2<0.
\]
Since $(x-\sqrt2)^3<0$ on $[-1,1]$, it follows that
$F'(x)>0$ throughout the interval. Thus $F$ is strictly increasing
and attains its minimum at $x=-1$.
Therefore,
\[
\mathcal{R}(\varphi_{\mathrm L},\psi_R)
=
\sqrt{F(-1)}
=
8\sqrt2-11.
\]

Since the minimum is attained at $z_0=-1$, we have
$\psi_R(z_0)=2(\sqrt2-1)$.
Let $r_0=8\sqrt2-11$. Then
$\varphi_{\mathrm L}(-r_0)=\sqrt{1-r_0}
=2(\sqrt2-1)=\psi_R(-1)$.
Hence the radius is sharp.
\end{proof}	

The Ma--Minda starlike and convex classes associated with the cardioid-type function were introduced by Kumar and Kamaljeet~\cite{MR4216363}. These classes correspond to the function
\[
\psi_c(z)=1+ze^z.
\]
In particular,
$\mathcal{ST}_{c}=\mathcal{ST}(\psi_c)$ and
$\mathcal{CV}_{c}=\mathcal{CV}(\psi_c)$.
The following theorem determines the sharp inclusion radius of the
lemniscate-associated class in the cardioid-type class.

The following radius was obtained in
\cite[Theorem~3.7(i)]{MR4216363}.
We provide an analytical boundary-minimization proof.
\begin{theorem}
The $\mathcal{P}_{\psi_c}$-radius of the class $\mathcal{P}_{\varphi_{\mathrm{L}}}$ is
\[
\mathcal{R}(\varphi_{\mathrm{L}},\psi_c)=2e^{-1}-e^{-2} \approx 0.600424.
\]
\end{theorem}

\begin{proof}
Let $\psi_c(z)=1+ze^z$. By \eqref{R1-formula},
\begin{equation}\label{c-eq1}
\mathcal{R}(\varphi_{\mathrm{L}},\psi_c)
=
\min_{|z|=1}\left|\psi_c(z)^2-1\right|.
\end{equation}
Set $z=e^{it}$, $-\pi\le t\le\pi$. Since
$
\psi_c(e^{it})
=
1+e^{\cos t}e^{i(t+\sin t)},
$
we obtain
\begin{align*}
|\psi_c(e^{it})^2-1|^2
&=
\left|
2e^{\cos t}e^{i(t+\sin t)}
+
e^{2\cos t}e^{2i(t+\sin t)}
\right|^2\\
&=
e^{2\cos t}
\left|2+e^{\cos t}e^{i(t+\sin t)}\right|^2.
\end{align*}
Define,
$
F(t):=
|\psi_c(e^{it})^2-1|^2
=
e^{4\cos t}
+4e^{2\cos t}
+4e^{3\cos t}\cos(t+\sin t).
$
Since $F$ is even, it suffices to consider $t\in[0,\pi]$.
Let
$
\rho=e^{\cos t},
\,
\alpha=t+\sin t.
$
Then
\[
F(t)
=
\rho^2\left|2+\rho e^{i\alpha}\right|^2.
\]
Since $\cos t$ changes sign at $t=\pi/2$, we split the interval $[0,\pi]$ into
$[0,\pi/2]$ and $[\pi/2,\pi]$, corresponding respectively to
$\rho=e^{\cos t}\ge1$ and $\rho\le1$.
On $[\pi/2,\pi]$, we have $e^{-1}\le\rho\le1$. By the reverse triangle inequality,
$
\left|2+\rho e^{i\alpha}\right|\ge 2-\rho,
$
and hence
$
F(t)\ge \rho^2(2-\rho)^2=:g(\rho).
$
Differentiating,
$
g'(\rho)
=
4\rho(2-\rho)(1-\rho)\ge0,
\, e^{-1}\le\rho\le1.
$
Therefore $g$ is increasing on $[e^{-1},1]$, and
\[
F(t)
\ge
g(e^{-1})
=
e^{-2}(2-e^{-1})^2
=
(2e^{-1}-e^{-2})^2.
\]
Equality is attained at $t=\pi$, where $\rho=e^{-1}$ and
$\alpha=\pi$.

For $t\in[0,\pi/2]$, we have $\rho\ge1$ and
$0\le\alpha\le\pi/2+1<\pi$. Since $\cos\alpha$ is decreasing on
$[0,\pi]$,
\[
\cos\alpha
\ge
\cos\left(\frac{\pi}{2}+1\right)
=
-\sin1.
\]
Consequently,
\begin{align*}
F(t)
&=
\rho^2\left(4+\rho^2+4\rho\cos\alpha\right)\\
&\ge
\rho^2\left(4+\rho^2-4\rho\sin1\right)\\
&=
\rho^2\left((\rho-2\sin1)^2+4\cos^21\right)\\
&\ge
4\rho^2\cos^21.
\end{align*}
Since $\rho\ge1$ and $1<\pi/3$, we have
$\cos1>\cos(\pi/3)=1/2$, and therefore
$
F(t)>1,
\,
0\le t\le\pi/2.
$
On the other hand,
$
F(\pi)
=
(2e^{-1}-e^{-2})^2
<1.
$
Hence the global minimum occurs at $t=\pi$, and
\[
\mathcal{R}(\varphi_{\mathrm{L}},\psi_c)
=
\sqrt{F(\pi)}
=
2e^{-1}-e^{-2}.
\]
Since the minimum is attained at $z_0=-1$, we have
$\psi_c(z_0)=1-e^{-1}$.
Let $r_0=2e^{-1}-e^{-2}$. Then
$\varphi_{\mathrm L}(-r_0)=\sqrt{1-r_0}=1-e^{-1}=\psi_c(-1)$.
Hence the radius $r_0=2e^{-1}-e^{-2}$ is sharp.
\end{proof}

	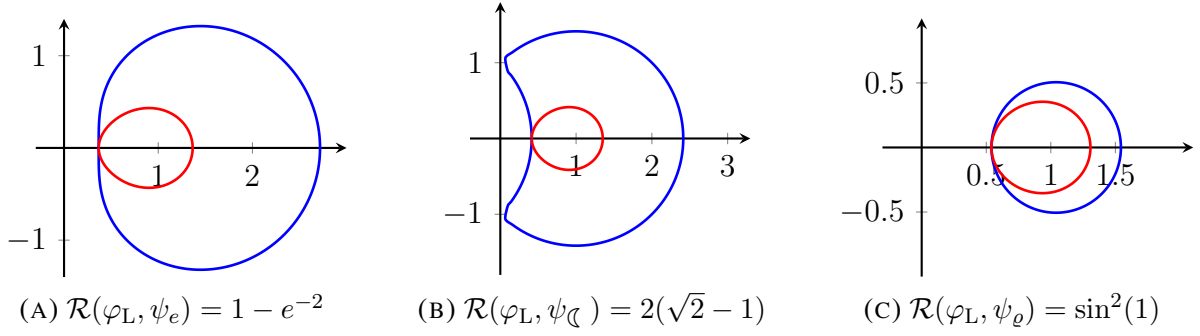
\begin{figure}[!htbp]
	\centering
	\begin{subfigure}{0.32\textwidth}
		\centering
		\begin{tikzpicture}
			\begin{axis}[
			    axis lines=center,
			    xmin=-0.3, xmax=3.0,
			    ymin=-1.4, ymax=1.4,
			    xtick={ 1.0, 2.0},
			    ytick={-1.0, 1.0},
			    thick,
			    scale=0.6
			]
			
			    \addplot[
			        domain=0:360,
			        samples=250,
			        smooth,
			        color=blue,
			        line width=1pt
			    ]
			    ({exp(cos(x))*cos(deg(sin(x)))}, {exp(cos(x))*sin(deg(sin(x)))});
			    
		        \def\R{1-exp(-2)}
		
		        \addplot[
		            domain=-180:180,
		            samples=400,
		            smooth,
		            red,
		            line width=1pt
		        ]
		        ({lemX(\R,x)}, {lemY(\R,x)});
			\end{axis}
		\end{tikzpicture}
		\caption{$\mathcal{R}(\varphi_{\mathrm L},\psi_e)=1-e^{-2}$}
	\end{subfigure}\hfill
	\begin{subfigure}{0.32\textwidth}
		\centering
		\begin{tikzpicture}
			\begin{axis}[
			    axis lines=center,
			    width=9cm,
			    height=8.8cm,
			    xmin=-0.4, xmax=3.3,
			    ymin=-1.8, ymax=1.8,
			    xtick={1.0,2.0,3.0},
			    ytick={-1,1},
			    thick,
			    scale=0.5
			]
			\addplot[
			    domain=-180:180,
			    samples=400,
			    smooth,
			    color=blue,
			    line width=1pt,
			]
			({cos(x) + sqrt(abs(cos(x)) + (cos(x))^2)},
			 {sin(x) + sign(sin(2*x)) * sqrt(abs(cos(x)) - (cos(x))^2)});
			 
			\def\R{0.828}
				
			\addplot[
			    domain=-180:180,
			    samples=400,
			    smooth,
			    red,
			    line width=1pt
			]
			({lemX(\R,x)}, {lemY(\R,x)});
			\end{axis}
		\end{tikzpicture}
		\caption{$\mathcal{R}(\varphi_{\mathrm{L}},\psi_{\leftmoon})=2(\sqrt2-1)$}
	\end{subfigure}\hfill
	\begin{subfigure}{0.32\textwidth}
		\centering
	    \begin{tikzpicture}
	    \begin{axis}[
	        axis equal, 
	        axis lines=center,
	        xmin=-0.2, xmax=2.0,
	        ymin=-1.0, ymax=1.0,
	        xtick={0, 0.5, 1.0, 1.5},
	        ytick={-0.5, 0.5},
	        thick,
	        scale=0.6,
	        declare function={
	            R = (sin(deg(1)))^2; 
	            mag(\t) = sqrt(1 + 2*R*cos(\t) + R^2);
	            phiX(\t) = sqrt((mag(\t) + 1 + R*cos(\t))/2);
	            phiY(\t) = sign(sin(\t)) * sqrt(max(0, (mag(\t) - (1 + R*cos(\t)))/2));
	            qX(\t) = ((exp(cos(\t/2)) + exp(-cos(\t/2)))/2) * cos(deg(sin(\t/2)));
	            qY(\t) = ((exp(cos(\t/2)) - exp(-cos(\t/2)))/2) * sin(deg(sin(\t/2)));
	        }
	    ]
	
	    \addplot[
	        domain=0:360,
	        samples=400,
	        smooth,
	        color=blue,
	        line width=1pt,
	        variable=\t
	    ] 
	    ({qX(\t)}, {qY(\t)});
	
	    \addplot[
	        domain=0:360,
	        samples=400,
	        smooth,
	        color=red,
	        line width=1pt,
	        variable=\t
	    ] 
	    ({phiX(\t)}, {phiY(\t)});
	
	    \end{axis}
	    \end{tikzpicture}
		\caption{$\mathcal{R}(\varphi_{\mathrm{L}},\psi_{\varrho})=\sin^2(1)$}
	\end{subfigure}
		\caption{Sharp radii associated with the lemniscate function
		$\varphi_{\mathrm{L}}$ for $\psi_{e}$, $\psi_{\leftmoon}$ and $\psi_{\varrho}$.}
		\label{p2f1}
	\end{figure}
	
\section{Sharp Radii for the Lune-Associated Class}
In this section, we determine the sharp $\mathcal P_{\psi}$-radii of
$\mathcal P_{\varphi_{\leftmoon}}$ for six choices of $\psi$.
The corresponding results for the associated Ma--Minda starlike and
convex classes follow from Corollary~\ref{cor:maminda-class-radius}.

Wani and Swaminathan
\cite{MR4129627} obtained the nonsharp lower bound
$
\displaystyle{\frac{\sqrt{17}-1}{6}}
$
for the $\mathcal{ST}(\psi_{\mathrm{Ne}})$-radius of
$\mathcal{ST}(\varphi_{\leftmoon})$ and posed the problem of finding
the sharp value. The following theorem solves this problem and gives
the exact radius $8/15$.

\begin{theorem}
The $\mathcal{P}_{\psi_{\mathrm{Ne}}}$-radius of the class $\mathcal{P}_{\varphi_{\leftmoon}}$ is
\[
\mathcal{R}(\varphi_{\leftmoon},\psi_{\mathrm{Ne}})=\frac{8}{15}\approx0.533333.
\]
\end{theorem}

\begin{proof}
Let $\psi_{\mathrm{Ne}}(z)=1+z-z^3/3$. By \eqref{R2-formula},
\[
\mathcal{R}(\varphi_{\leftmoon},\psi_{\mathrm{Ne}})
=
\min_{|z|=1}
\left|
\frac{\psi_{\mathrm{Ne}}(z)^2-1}
{2\psi_{\mathrm{Ne}}(z)}
\right|.
\]
Setting $z=e^{i\theta}$ and $x=\cos\theta\in[-1,1]$,
using the Chebyshev polynomial identities, we obtain
$
|\psi_{\mathrm{Ne}}(e^{i\theta})|^2
=(1/9)\left(25+36x-12x^2-24x^3\right).
$
By \eqref{nep-eq2}
$
|\psi_{\mathrm{Ne}}(e^{i\theta})^2-1|^2
=(16/81)(4-3x^2)(13+18x-3x^2-12x^3).
$
Using the above expressions, we define $F(x)$ over $[-1, 1]$
\[
F(x):=
\left|
\frac{\psi_{\mathrm{Ne}}(e^{i\theta})^2-1}
{2\psi_{\mathrm{Ne}}(e^{i\theta})}
\right|^2
=
\frac{4(4-3x^2)(13+18x-3x^2-12x^3)}
{9(25+36x-12x^2-24x^3)}.
\]
To show that the minimum is $64/225$, consider
\[
F(x)-\frac{64}{225}
=
\frac{4(x-1)A(x)}{75B(x)},
\]
where
$
A(x)=300x^4+375x^3-347x^2-708x-300
$
and
$
B(x)=25+36x-12x^2-24x^3.
$
We first establish that $A(x)<0$ on $[-1,1]$.
A Sturm sequence for $A$, with positive rescaling of the
successive remainders, is
\[
\begin{aligned}
S_0(x)&=300x^4+375x^3-347x^2-708x-300,\\
S_1(x)&=1200x^3+1125x^2-694x-708,\\
S_2(x)&=16729x^2+30514x+15660,\\
S_3(x)&=-269171257x-629339424,\\
S_4(x)&=-1.
\end{aligned}
\]
Writing $V(x)$ for the number of sign variations
in the Sturm sequence, the endpoint signs are
\[
\begin{array}{c|ccccc|c}
x&S_0&S_1&S_2&S_3&S_4&V(x)\\ \hline
-1&-&-&+&-&-&2\\
 1&-&+&+&-&-&2
\end{array}
\]
Thus Sturm's theorem gives no zeros of $A$ in $(-1,1)$.
Since $A(-1)=-14$ and $A(1)=-680$, we conclude that
$A(x)<0$ throughout $[-1,1]$.
For the denominator polynomial, we have
\[
B'(x)=-12(6x^2+2x-3),
\]
whose zeros are
\[
x_-=\frac{-1-\sqrt{19}}6,
\qquad
x_+=\frac{-1+\sqrt{19}}6.
\]
Both belong to $[-1,1]$, and
\[
B(-1)=1,\qquad B(1)=25,\qquad
B(x_\pm)=\frac{169\pm38\sqrt{19}}9.
\]
All four values are positive, because
$
169^2-(38\sqrt{19})^2=1125>0.
$
Consequently, $B(x)>0$ throughout $[-1,1]$.
Since $x-1\le0$, $A(x)<0$, and $B(x)>0$, we obtain
\[
F(x)-\frac{64}{225}\ge0,
\qquad -1\le x\le1,
\]
with equality only at $x=1$. Therefore,
\[
\mathcal{R}(\varphi_{\leftmoon},\psi_{\mathrm{Ne}})
=
\sqrt{F(1)}
=
\frac{8}{15}.
\]
Since the minimum is attained at $z_0=1$, we have
$\psi_{\mathrm{Ne}}(z_0)=5/3$.
Let $r_0=8/15$. Then
$\varphi_{\leftmoon}(r_0)=5/3=\psi_{\mathrm{Ne}}(z_0)$.
Hence the radius $r_0=8/15$ is sharp.
\end{proof}

We now consider the rational function
$\psi_R(z)=1+(z^2+kz)/(k^2-kz)$, where $k=1+\sqrt2$.
The following theorem determines the sharp
$\mathcal P_{\psi_R}$-radius of
$\mathcal P_{\varphi_{\leftmoon}}$.
The minimum is established by factoring the difference
between the boundary expression and its value at $z=-1$.

\begin{theorem}
The $\mathcal{P}_{\psi_{R}}$-radius of the class $\mathcal{P}_{\varphi_{\leftmoon}}$ is
\[
\mathcal{R}(\varphi_{\leftmoon},\psi_{R})=\frac{5-3\sqrt2}{4} \approx0.18934.
\]
\end{theorem}

\begin{proof}

Let
$
\psi_R(z)
=1+(z^2+kz)/(k^2-kz),\;k=1+\sqrt2.
$
By \eqref{R2-formula},
\[
\mathcal R(\varphi_{\leftmoon},\psi_R)
=
\min_{|z|=1}
\left|
\frac{\psi_R(z)^2-1}
{2\psi_R(z)}
\right|.
\]
Let $z=e^{i\theta}$ and $x=\cos \theta \in [-1,1]$, using the Chebyshev polynomial identities, we obtain
\begin{align}\label{lune-rat-eq1}
|\psi_R(e^{i\theta})|^2
=
\frac{2(1+x^2)}
{(1+\sqrt2)(\sqrt2-x)}
\end{align}
and
\[
|\psi_R(e^{i\theta})^2-1|^2
=
\frac{(x+\sqrt2)
\left(4x^2-(1+3\sqrt2)x+6+3\sqrt2\right)}
{(1+\sqrt2)^3(x-\sqrt2)^2}.
\]
We define
\[
F(x):=
\left|
\frac{\psi_R(e^{i\theta})^2-1}
{2\psi_R(e^{i\theta})}
\right|^2
=
\frac{(x+\sqrt2)
\left(4x^2-(1+3\sqrt2)x+6+3\sqrt2\right)}
{8(3+2\sqrt2)(1+x^2)(\sqrt2-x)}.
\]
To minimize $F(x)$, we isolate its nonconstant factor
by defining
\[
g(x) = \frac{(x+\sqrt{2})\left(4x^2-(1+3\sqrt{2})x+6+3\sqrt{2}\right)}{(1+x^2)(\sqrt{2}-x)}, \quad x \in [-1, 1].
\]
Evaluating $g(x)$ at the boundary $x = -1$ gives $g(-1) = (9-4\sqrt{2})/2$. The difference is
\[
g(x) - g(-1) = \frac{(x+1)\left[(17-4\sqrt{2})x^2-(11+3\sqrt{2})x+(20+3\sqrt{2})\right]}{2(1+x^2)(\sqrt{2}-x)}.
\]
Let $A(x)=(17-4\sqrt{2})x^2-(11+3\sqrt{2})x+(20+3\sqrt{2})$. The leading coefficient $17-4\sqrt{2}$ is positive, and the discriminant $\Delta = -1125+182\sqrt{2}$ is strictly negative. Therefore, $A(x) > 0$ for all real $x$.
Since $x \in [-1, 1]$, we have $x+1 \ge 0$, which guarantees $g(x) - g(-1) \ge 0$. Consequently, $g(x)$ and $F(x)$ attain their absolute minimum at $x = -1$. Evaluating the minimal value yields
\[
F(-1) = \frac{57 + 40\sqrt{2}}{16(99 + 70\sqrt{2})} = \frac{(5 + 4\sqrt{2})^2}{16(7 + 5\sqrt{2})^2}.
\]
Hence the radius is
\[
\mathcal{R}(\varphi_{\leftmoon},\psi_{R})
= \sqrt{F(-1)} = \frac{5 + 4\sqrt{2}}{4(7 + 5\sqrt{2})} = \frac{5 - 3\sqrt{2}}{4}.
\]

Since the minimum is attained at $z_0=-1$, we have
$\psi_R(z_0)=2(\sqrt2-1)$.
Let $r_0=(5-3\sqrt2)/4$. Then
$\varphi_{\leftmoon}(-r_0)
=-r_0+\sqrt{1+r_0^2}
=2(\sqrt2-1)
=\psi_R(z_0)$.
Hence the radius is sharp.
\end{proof}

	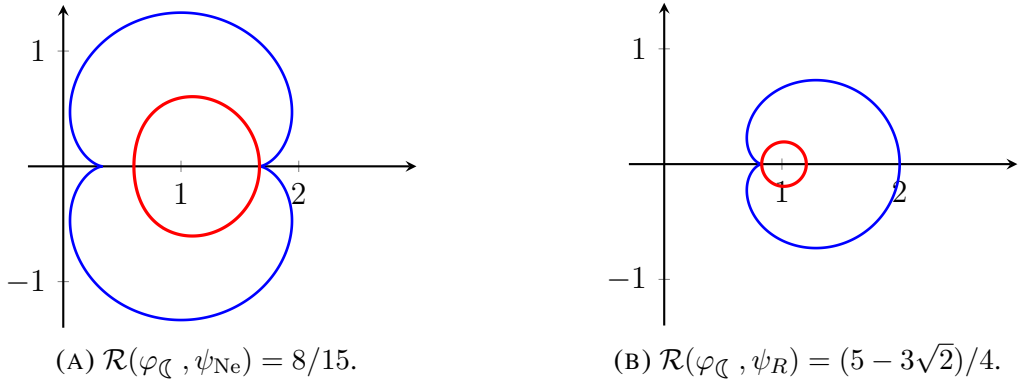
\begin{figure}[!htbp]
	\centering
	\begin{subfigure}{0.48\textwidth}
	\centering
	\begin{tikzpicture}
		\begin{axis}[
			axis lines=center,
			xmin=-0.3, xmax=3.0,
			ymin=-1.4, ymax=1.4,
		    xtick={ 1.0, 2.0},
		    ytick={-1.0, 1.0},
		    thick,
		    scale=0.75
			]
	
    		\addplot[
    		    domain=-180:180,
    		    samples=400,
    		    smooth,
    		    color=blue,
    		    line width=1pt,
    		]
    		({1 + cos(x) - (1/3) * cos(3*x)},
    		 {sin(x) - (1/3) * sin(3*x)});
	\def\R{8/15}
		\addplot[
		    red,
		    domain=-180:180,
		    samples=500,
		    smooth,
		    line width=1.2pt
		]
		({luneX(\R,x)}, {luneY(\R,x)});
		\end{axis}
	\end{tikzpicture}
	\caption{$\mathcal R(\varphi_{\leftmoon},\psi_{\mathrm{Ne}})
	=8/15$.}

	\end{subfigure}
	\begin{subfigure}{0.48\textwidth}
	\centering
	\begin{tikzpicture}
		\begin{axis}[
			axis lines=center,
			xmin=-0.3, xmax=3.0,
			ymin=-1.4, ymax=1.4,
		    xtick={ 1.0, 2.0},
		    ytick={-1.0, 1.0},
		    thick,
		    scale=0.75
			]
	
   			\addplot[
   			    domain=-180:180,
   			    samples=400,
   			    smooth,
   			    color=blue,
   			    line width=1pt,
   			]
   			({(cos(2*x) - 2*sqrt(2)*cos(x) + 3 + 2*sqrt(2)) / den(x)},
   			 {(sin(2*x) + 2*sin(x)) / den(x)});	
	\def\R{(5-3*sqrt(2))/4}
		\addplot[
		    red,
		    domain=-180:180,
		    samples=500,
		    smooth,
		    line width=1.2pt
		]
		({luneX(\R,x)}, {luneY(\R,x)});
		\end{axis}
	\end{tikzpicture}
	\caption{$\mathcal R(\varphi_{\leftmoon},\psi_R)
	=(5-3\sqrt2)/4$.}
	\end{subfigure}
	\caption{Sharp radii associated with the lune function
	$\varphi_{\leftmoon}$ for $\psi_{\mathrm{Ne}}$ and $\psi_R$.}
	\label{fig:lune-radii}
	\label{p2f2}
	\end{figure}

\begin{theorem}
The $\mathcal{P}_{\psi_{car}}$-radius of the class $\mathcal{P}_{\varphi_{\leftmoon}}$ is
\[
\mathcal{R}(\varphi_{\leftmoon},\psi_{car})=\frac34=0.75.
\]
\end{theorem}

\begin{proof}

Let $\psi_{car}(z) = 1 +z+z^2/2$, then by \eqref{R2-formula}
\[
\mathcal{R}(\varphi_{\leftmoon},\psi_{car})
= \min_{|z|=1} \left| \frac{\psi_{car}(z)^2 - 1}{2\psi_{car}(z)} \right|.
\]
Set $z=e^{i\theta}$ and $ x = \cos\theta \in [-1, 1]$. A direct computaions gives $|\psi_{car}(e^{i\theta})|^2 = (1/4)(8x^2+12x+5).$
Using \eqref{Acar-eq2},
$
|\psi_{car}(e^{i\theta})^2 - 1|^2 =(1/16)(64x^3+160x^2+152x+65).
$
Combining these, we define $F(x)$ over $[-1, 1]$
\begin{equation}
F(x) = \left| \frac{\psi_{car}(e^{i\theta})^2 - 1}{2\psi_{car}(e^{i\theta})} \right|^2 = \frac{64x^3 + 160x^2 + 152x + 65}{16(8x^2 + 12x + 5)}.
\end{equation}
To show that $F(x)\ge 9/16 $ on $[-1,1]$, consider the difference $F(x)-9/16$.
\[
F(x) - \frac{9}{16}
= \frac{4(x + 1)(16x^2 + 6x + 5)}{16(8x^2 + 12x + 5)}.
\]
Since $x+1\ge0$ on $[-1,1]$, and the quadratics
$16x^2+6x+5$ and $8x^2+12x+5$ have negative discriminants
$-284$ and $-16$, respectively, with positive leading coefficients, both are strictly positive on $\mathbb{R}$. Hence
\[
F(x)-\frac{9}{16}\ge0
\quad\Longrightarrow\quad
F(x)\ge\frac{9}{16},
\qquad x\in[-1,1].
\]
The absolute minimum is attained if and only if $x = -1$. then the radius is
\[
\mathcal{R}(\varphi_{\leftmoon},\psi_{car})
= \sqrt{F(-1)} = \sqrt{\frac{9}{16}} = \frac{3}{4}.
\]
Since the minimum is attained at $x=-1$, we have
$\theta=\pm\pi$ and hence $z_0=e^{i\theta}=-1$.
Let $r_0=3/4$. At the boundary point $z_0=-1$,
$\psi_{car}(z_0)=1/2.$
Also, $\varphi_{\leftmoon}(-r_0)=1/2$.
The radius is sharp.
\end{proof}

The following radius is the case $n=4$ of
\cite[Theorem~5.3(c)]{MR4510346}.
Here we establish the boundary minimum through an
explicit factorization and polynomial positivity estimates.

For the three-leaf function
$\psi_{3L}(z)=1+4z/5+z^4/5$, the boundary minimum
is attained at nonreal points of the unit circle.
The following theorem determines the sharp
$\mathcal P_{\psi_{3L}}$-radius of
$\mathcal P_{\varphi_{\leftmoon}}$
by an exact polynomial factorization and positivity estimates.

\begin{theorem}
The $\mathcal{P}_{\psi_{3L}}$-radius of the class $\mathcal{P}_{\varphi_{\leftmoon}}$ is
\[
\mathcal{R}(\varphi_{\leftmoon},\psi_{3L})=\frac{3\sqrt{139}}{70} \approx 0.505278.
\]
\end{theorem}

\begin{proof}
For $\psi_{3L}(z) = 1 + (4/5)z + (1/5)z^4$ then by \eqref{R2-formula}
\[
\mathcal{R}(\varphi_{\leftmoon},\psi_{3L})
= \min_{|z|=1} \left| \frac{\psi_{3L}(z)^2 - 1}{2\psi_{3L}(z)} \right|.
\]
Let $z=e^{i\theta}$ and   $x = \cos\theta \in [-1, 1]$, A direct computation, using the Chebyshev polynomial identities, yields $|\psi_{3L}(e^{i\theta})|^2 = (4/25)(20x^4 + 8x^3 - 20x^2 + 4x + 13),$ and using \eqref{3L-eq2},
\[
|\psi_{3L}(e^{i\theta})^2 - 1|^2 = \frac{1}{625}(32x^3- 24x +17)(160x^4 + 32x^3- 160x^2+ 56x +137).
\]
Combining these, we define the function $F(x)$ over $[-1, 1]$
\[
F(x)=
\left|
\frac{\psi_{3L}(e^{i\theta})^2-1}
{2\psi_{3L}(e^{i\theta})}
\right|^2
=
\frac{(32x^3-24x+17)
(160x^4+32x^3-160x^2+56x+137)}
{400(20x^4+8x^3-20x^2+4x+13)}.
\]
To show that the minimum of $F$ is $1251/4900$, consider
\[
F(x)-\frac{1251}{4900}
=
\frac{(2x-1)^2A(x)}
{19600B(x)},
\]
where
$
A(x)
=
62720x^5+75264x^4-50176x^3-48148x^2+61796x+49069
$
and
$
B(x)
=
20x^4+8x^3-20x^2+4x+13.
$
Differentiating $A$, we obtain
$
A'(x)
=
313600x^4+301056x^3-150528x^2-96296x+61796.
$
Moreover,
$
25A'(x)
=
(2800x^2+1344x-1180)^2
+
(1038464x^2+764440x+152500).
$
The quadratic in the second term has discriminant
$
\Delta=-49094526400<0
$
and positive leading coefficient; hence $A'(x)>0$ on $\mathbb R$.
Thus $A$ is strictly increasing, and since
$
A(-1)=1845>0,
$
we have $A(x)>0$ for $x\in[-1,1]$.
Moreover,
$
B(x)
=
(4x^2-3)^2+4x^2(x+1)^2+4(x+1).
$
For $-1<x\le1$, the last term is positive and the
remaining terms are nonnegative. At $x=-1$, we have
$B(-1)=1$. Hence $B(x)>0$ throughout $[-1,1]$.
Therefore,
\[
F(x)-\frac{1251}{4900}\ge0,
\qquad x\in[-1,1],
\]
with equality only when $2x-1=0$, that is, at $x=1/2$. Hence
\[
\mathcal{R}(\varphi_{\leftmoon},\psi_{3L})
=
\sqrt{F\left(\frac12 \right) }
=
\frac{3\sqrt{139}}{70}.
\]

Since the minimum is attained at $x=1/2$, we take $\theta_0=\pi/3$ and set
$z_0=e^{i\pi/3}$. Then
$\psi_{3L}(z_0)=13/10+(3\sqrt3/10)i$.
Let $r_0=3\sqrt{139}/70$ and
$w_0=\varphi_{\leftmoon}^{-1}(\psi_{3L}(z_0))
=78/245+(111\sqrt3/490)i$.
Since $|w_0|^2=1251/4900=r_0^2$, we have $|w_0|=r_0$ and
$\varphi_{\leftmoon}(w_0)=\psi_{3L}(z_0)$.
Hence the radius is sharp.
\end{proof}

For the exponential function $\psi_e(z)=e^z$,
the inverse composition simplifies to
$\varphi_{\leftmoon}^{-1}(\psi_e(z))=\sinh z$.
The following theorem evaluates its minimum modulus
on the unit circle and obtains the sharp
$\mathcal P_{\psi_e}$-radius of
$\mathcal P_{\varphi_{\leftmoon}}$.
The following radius was obtained by Devadiga et al.\
\cite[Theorem~8.1]{DevadigaEtAl2026Atlas}.
We provide an alternative proof by minimizing
$|\sinh z|$ on the unit circle.
\begin{theorem}
The $\mathcal{P}_{\psi_{e}}$-radius of the class $\mathcal{P}_{\varphi_{\leftmoon}}$ is
\[
\mathcal{R}(\varphi_{\leftmoon},\psi_{e})=\sin(1) \approx 0.841471.
\]
\end{theorem}

\begin{proof}
Let $\psi_e(z) = e^z$ by \eqref{R2-formula}
\[
\mathcal{R}(\varphi_{\leftmoon},\psi_{e})
= \min_{|z|=1} \left| \frac{\psi_e(z)^2 - 1}{2\psi_e(z)} \right| 
= \min_{|z|=1} \left| \frac{e^{2z} - 1}{2e^z} \right| = \min_{|z|=1} |\sinh z|.
\]
Thus, the objective is to minimize the modulus of $f(z) = \sinh z$. Set $z=e^{i\theta}$, where $-\pi\le \theta\le\pi$. Using
$\sinh(x+iy)=\sinh x\cos y+i\cosh x\sin y$, we obtain
\[
|\sinh(x+iy)|^2=\sinh^2x+\sin^2y.
\]
With $x=\cos \theta$ and $y=\sin \theta$, define
\begin{equation}
    F(\theta) = \sinh^2(\cos \theta) + \sin^2(\sin \theta).
\end{equation}
Since $F(-\theta)=F(\theta)$ and $F(\pi-\theta)=F(\theta)$, it is enough to consider
$0\le \theta\le\pi/2$. Writing
$v=\sin \theta$ and $u=\cos \theta=\sqrt{1-v^2}$, where
$0\le v\le1$, the function becomes
\begin{equation}
    H(v) = \sinh^2\left(\sqrt{1-v^2}\right) + \sin^2(v).
\end{equation}
To find the critical points, we differentiate $H(v)$ with respect to $v$
\begin{align*}
    H'(v) &= 2\sinh\left(\sqrt{1-v^2}\right)\cosh\left(\sqrt{1-v^2}\right)\left(\frac{-v}{\sqrt{1-v^2}}\right) + 2\sin(v)\cos(v) \\
    &= -\frac{v}{\sqrt{1-v^2}}\sinh\left(2\sqrt{1-v^2}\right) + \sin(2v) \\
    &= \sin(2v) - \frac{v}{u}\sinh(2u).
\end{align*}
For $0<v<1$, we have $u=\sqrt{1-v^2}>0$.
The elementary inequalities $\sin \theta<\theta$ and $\sinh \theta>\theta$
for $\theta>0$ give
\[
\sin(2v)<2v,
\qquad
\frac{v}{u}\sinh(2u)>2v.
\]
Therefore, $H'(v) < 0$ for all $v \in (0, 1)$. $H(v)$ is strictly decreasing on $(0, 1)$. The minimum must occur at the boundary $v = 1$, which corresponds strictly to $\sin \theta = 1$, or $\theta = \pi/2$ (and by symmetry, $\theta = -\pi/2$). Evaluating the function at this minimum yields:
\begin{equation}
    F\left(\pm\frac{\pi}{2}\right) = \sinh^2(0) + \sin^2(\pm 1) = \sin^2(1).
\end{equation}
Taking the principal square root gives the exact minimum radius:
\[
  r_0 = \sqrt{\sin^2(1)} = \sin(1).
\]
Since the minimum is attained at $\theta_0=\pi/2$, we set
$z_0=e^{i\pi/2}=i$. Then $\psi_e(z_0)=e^i$.
Let $r_0=\sin(1)$ and $w_0=i\sin(1)$, so that $|w_0|=r_0$.
Also, $\varphi_{\leftmoon}(w_0)=e^i=\psi_e(z_0)$.
Hence the radius is sharp.
\end{proof}

Finally, we consider the lemniscate function
$\psi_{\mathrm L}(z)=\sqrt{1+z}$.
The inverse composition
$z/(2\sqrt{1+z})$ is unbounded at $z=-1$,
so Theorem~\ref{thm:general-radius} does not apply directly.
The following theorem determines the sharp
$\mathcal P_{\psi_{\mathrm L}}$-radius of
$\mathcal P_{\varphi_{\leftmoon}}$
using Remark~\ref{rem:boundary-pole}.
The following result was established by Devadiga et al.\
\cite[Theorem~7.11]{DevadigaEtAl2026Atlas}.

\begin{theorem}
The $\mathcal P_{\psi_{\mathrm L}}$-radius of the class
$\mathcal P_{\varphi_{\leftmoon}}$ is
\[
\mathcal R(\varphi_{\leftmoon},\psi_{\mathrm L})
=
\frac{1}{2\sqrt2}
\approx0.353553.
\]
\end{theorem}

\begin{proof}

For $\psi_{\mathrm L}(z)=\sqrt{1+z}$, we have
\[
\varphi_{\leftmoon}^{-1}(\psi_{\mathrm L}(z))
=
\frac{z}{2\sqrt{1+z}}.
\]
The function is continuous on
$\overline{\mathbb D}\setminus\{-1\}$ and
\[
\left|
\frac{z}{2\sqrt{1+z}}
\right|
\to\infty
\qquad\text{as }z\to-1.
\]
Hence, by Remark~\ref{rem:boundary-pole},
\[
\mathcal R(\varphi_{\leftmoon},\psi_{\mathrm L})
=
\min_{\substack{|z|=1\\ z\ne-1}}
\left|
\frac{z}{2\sqrt{1+z}}
\right|.
\]
Set $z=e^{i\theta}$, where $-\pi<\theta<\pi$. Since
$
|1+e^{i\theta}|^2
=
(1+\cos \theta)^2+\sin^2\theta
=
2+2\cos \theta,
$
we obtain
\[
\left|
\frac{1}{2\sqrt{1+e^{i\theta}}}
\right|
=
\frac{1}{2\sqrt{|1+e^{i\theta}|}}
=
\frac{1}{2(2+2\cos \theta)^{1/4}}.
\]
Therefore, minimizing the above expression is equivalent to maximizing
$2+2\cos \theta$. Since $\cos \theta\le1$, the maximum is attained at $\theta=0$.
Thus
\[
\mathcal R(\varphi_{\leftmoon},\psi_{\mathrm L})
=
\frac{1}{2(4)^{1/4}}
=
\frac{1}{2\sqrt2}.
\]
The minimum is attained at $\theta_0=0$, so that $z_0=1$ and
$\psi_{\mathrm L}(z_0)=\sqrt2.$ Let $r_0=1/2\sqrt2.$ Then
$\varphi_{\leftmoon}(r_0)=r_0+\sqrt{1+r_0^2}=\sqrt2=\psi_{\mathrm L}(1).$
Hence the radius $r_0$ is sharp.
\end{proof}

\section{Sharp Radii for the Lima\c{c}on-Associated Class}

In this section, we determine the sharp $\mathcal P_{\psi}$-radii of
$\mathcal P_{\varphi_{\mathrm{Lim}}}$ for seven choices of $\psi$.
The corresponding results for the associated Ma--Minda starlike and
convex classes follow from Corollary~\ref{cor:maminda-class-radius}.
Using \eqref{R3-formula}, we evaluate the corresponding boundary
minima analytically. For later use, we record the rationalization
identity
\begin{equation}\label{eq:lim-rationalization}
\left|\sqrt{\psi(z)}-1\right|
=
\frac{|\psi(z)-1|}
{\left|\sqrt{\psi(z)}+1\right|}.
\end{equation}
Since
$\left|\sqrt{\psi(z)}+1\right|
\le 1+\sqrt{|\psi(z)|}$,
we obtain
\begin{equation}\label{eq:lim-rationalization-bound}
\left|\sqrt{\psi(z)}-1\right|
\ge
\frac{|\psi(z)-1|}
{1+\sqrt{|\psi(z)|}}.
\end{equation}
Moreover, setting $A=|\psi(z)-1|$ and using
$|\psi(z)|\le 1+A$, we get
\begin{equation}\label{eq:lim-A-bound}
\left|\sqrt{\psi(z)}-1\right|
\ge
\frac{A}{1+\sqrt{1+A}}
=\sqrt{1+A}-1.
\end{equation}
These estimates will be used in the subsequent proofs. In each case,
we establish sharpness by identifying an appropriate boundary contact
point. The corresponding sharp radii for the associated
Ma--Minda starlike and convex classes follow from
Corollary~\ref{cor:maminda-class-radius}.

We first consider the cardioid function
$\psi_C(z)=1+4z/3+2z^2/3$.
Rationalizing $\sqrt{\psi_C(z)}-1$ allows the boundary
minimum to be determined through modulus estimates,
with equality at $z=-1$.
The following theorem gives the sharp
$\mathcal P_{\psi_C}$-radius of
$\mathcal P_{\varphi_{\mathrm{Lim}}}$.

\begin{theorem}
The $\mathcal{P}_{\psi_C}$-radius of the class $\mathcal{P}_{\varphi_{\mathrm{Lim}}}$ is
\[
\mathcal{R}(\varphi_{\mathrm{Lim}},\psi_C)=\sqrt2\left(1- \frac{1}{\sqrt 3}\right) \approx0.597717. 
\]
\end{theorem}

\begin{proof}
Let $\psi_C(z)=1+4z/3+2z^2/3$. By \eqref{R3-formula},
\[
\mathcal R(\varphi_{\mathrm{Lim}},\psi_C)
=
\sqrt2\min_{|z|=1}
\left|\sqrt{\psi_C(z)}-1\right|.
\]
Set $z=e^{i\theta}$ and $x=\cos\theta\in[-1,1]$.
By \eqref{eq:lim-rationalization-bound},
\[
\left|\sqrt{\psi_C(z)}-1\right|
\ge
\frac{|\psi_C(z)-1|}
{1+\sqrt{|\psi_C(z)|}}.
\]
Since
$
\psi_C(z)-1=(2/3)z(z+2),
$
we have
$
|\psi_C(z)-1|=(2/3)|z+2|=(2/3)\sqrt{5+4x}.
$
Moreover,
$
\psi_C(z)
=
(1/3)\left(1+2(z+1)^2\right),
$
and hence, by the triangle inequality,
\[
|\psi_C(z)|
\le
\frac{1+2|z+1|^2}{3}
=
\frac{5+4x}{3}.
\]
Therefore,
\[
\left|\sqrt{\psi_C(z)}-1\right|
\ge
\frac{2}{3}\,
\frac{\sqrt{5+4x}}
{1+\sqrt{(5+4x)/3}}.
\]
Set $u = \sqrt{(5 + 4x)/3}$. Since $x \in [-1, 1]$, it follows that
$u \in \left[ 1/\sqrt3, \sqrt3\, \right]$. Then we have
\[
\left|\sqrt{\psi_C(z)}-1\right|
\ge
\frac{2}{3}\,
\frac{\sqrt{5+4x}}
{1+\sqrt{(5+4x)/3}}
=\frac{2}{\sqrt{3}} \left( \frac{u}{1 + u} \right).
\]
The function $h(u) = u/(1+u)$ is strictly increasing for $u > 0$. Thus, its minimum on the interval $\left[ 1/\sqrt3, \sqrt3 \,\right]$ occurs at the left endpoint, $u = 1/\sqrt{3}$, which corresponds to $x = -1$. Substituting this value into $h(u)$ yields
\[
\left| \sqrt{\psi_C(z)} - 1 \right| \ge \frac{2}{\sqrt{3}} \left( \frac{1/\sqrt{3}}{1 + 1/\sqrt{3}} \right) = \frac{2}{3 + \sqrt{3}} = 1 - \frac{1}{\sqrt{3}}.
\]
At $z=-1$, we have $\psi_C(-1)=1/3$, and hence
$\left|\sqrt{\psi_C(-1)}-1\right|=1-1/\sqrt3$.
Thus the lower bound is attained at $z=-1$, and therefore
\[
\mathcal{R}(\varphi_{\mathrm{Lim}},\psi_C)
=
\sqrt2\left(1-\frac{1}{\sqrt3}\right).
\]
Since the minimum is attained at $z_0=-1$, we have
$\psi_C(z_0)=1/3$.
Let $r_0=\sqrt2\left(1-1/\sqrt3\right)$.
Then $\varphi_{\mathrm{Lim}}(-r_0)=1/3=\psi_C(z_0)$.
Hence the radius is sharp.
\end{proof}

For the nephroid function
$\psi_{\mathrm{Ne}}(z)=1+z-z^3/3$,
we estimate $|\sqrt{\psi_{\mathrm{Ne}}(z)}-1|$
in terms of $|\psi_{\mathrm{Ne}}(z)-1|$.
The resulting bound is attained at $z=1$ and yields
the sharp $\mathcal P_{\psi_{\mathrm{Ne}}}$-radius of
$\mathcal P_{\varphi_{\mathrm{Lim}}}$
in the following theorem.

\begin{theorem}
The $\mathcal{P}_{\psi_{\mathrm{Ne}}}$-radius of the class
$\mathcal{P}_{\varphi_{\mathrm{Lim}}}$ is
\[
\mathcal{R}(\varphi_{\mathrm{Lim}},\psi_{\mathrm{Ne}})
=
\sqrt2\left(\sqrt{\frac{5}{3}}-1\right)\approx0.411528.
\]
\end{theorem}

\begin{proof}

For $\psi_{\mathrm{Ne}}(z)=1+z-z^3/3$, it follows from
\eqref{R3-formula} that
\[
\mathcal R(\varphi_{\mathrm{Lim}},\psi_{\mathrm{Ne}})
=
\sqrt2\min_{|z|=1}
\left|\sqrt{\psi_{\mathrm{Ne}}(z)}-1\right|.
\]
By using \eqref{eq:lim-rationalization-bound},
\[
\left|\sqrt{\psi_{\mathrm{Ne}}(z)}-1\right|
\ge
\frac{|\psi_{\mathrm{Ne}}(z)-1|}{1+\sqrt{|\psi_{\mathrm{Ne}}(z)|}}.
\]
Let $A=|\psi_{\mathrm{Ne}}(z)-1|$, then, by \eqref{eq:lim-A-bound},
\begin{equation}\label{lim-nep-eq1}
\left|\sqrt{\psi_{\mathrm{Ne}}(z)}-1\right|
\ge
\sqrt{1+A}-1.
\end{equation}
We now determine the minimum of $A$ on $|z|=1$. Set $z=e^{i\theta}$ and $x=\cos\theta\in[-1,1]$. using $z^2+\bar z^{\,2}=4x^2-2$, we obtain
\[
A
=
\left|1-\frac{z^2}{3}\right| \implies
A^2
=
1+\frac{1}{9}-\frac{1}{3}(z^2+\bar z^{\,2})
=
\frac{4}{9}(4-3x^2).
\]
The minimum of $A^2$ occurs when $\cos^2 \theta = 1$, giving $A^2 \ge 4/9 \implies A \ge 2/3$.
Since the function $f(A) = \sqrt{1+A} - 1$ is strictly increasing for $A > 0$, we substitute $A = 2/3$ into \eqref{lim-nep-eq1}
\[
\left| \sqrt{\psi_\mathrm{Ne}(z)} - 1 \right| \ge \sqrt{1 + \frac{2}{3}} - 1 = \sqrt{\frac{5}{3}} - 1.
\]
The lower bound is attained at $z=1$. At this point, we have $\psi_\mathrm{Ne}(1) = 5/3$. Evaluating the modulus gives
\[
\left| \sqrt{\psi_\mathrm{Ne}(1)} - 1 \right| = \sqrt{\frac{5}{3}} - 1.
\]
Thus the lower bound is attained at $z=1$, and therefore
\[
\mathcal{R}(\varphi_{\mathrm{Lim}},\psi_\mathrm{Ne})
= \sqrt{2} \left( \sqrt{\frac{5}{3}} - 1 \right).
\]
Since the minimum is attained at $z_0=1$, we have
$\psi_{\mathrm{Ne}}(z_0)=5/3$.
Let
$r_0=\sqrt2\left(\sqrt{5/3}-1\right)$.
Then
$\varphi_{\mathrm{Lim}}(r_0)
=\left(1+r_0/\sqrt2\right)^2
=5/3
=\psi_{\mathrm{Ne}}(z_0)$.
Hence the radius is sharp.
\end{proof}

We next consider the rational function
\[
\psi_R(z)
=
1+\frac{z^2+kz}{k^2-kz},
\qquad
k=1+\sqrt2.
\]
The required boundary minimum is obtained from
\eqref{eq:lim-rationalization-bound} by establishing the
monotonicity of an associated real-valued function on $[-1,1]$.

\begin{theorem}
The $\mathcal P_{\psi_R}$-radius of the class
$\mathcal P_{\varphi_{\mathrm{Lim}}}$ is
\[
\mathcal R(\varphi_{\mathrm{Lim}},\psi_R)
=
\sqrt2\left(1-\sqrt{2\sqrt2-2}\right)
\approx0.127025.
\]
\end{theorem}

\begin{proof}
Let
$
\psi_R(z)
=1+(z^2+kz)/(k^2-kz),\; k=1+\sqrt2.
$
By \eqref{R3-formula},
\[
\mathcal R(\varphi_{\mathrm{Lim}},\psi_R)
=
\sqrt2\min_{|z|=1}
\left|\sqrt{\psi_R(z)}-1\right|.
\]
For
$
\psi_R(z)=\dfrac{k^2+z^2}{k^2-kz},
$
put $z=e^{i\theta}$ and $x=\cos\theta\in[-1,1]$. Then,
using \eqref{lune-rat-eq1}, we obtain
\begin{equation}\label{eq:R-V}
V(x):=|\psi_R(e^{i\theta})|^2
=
2(\sqrt2-1)
\frac{1+x^2}{\sqrt2-x}.
\end{equation}
since, $\psi_R(z)-1=(z^2+kz)/(k^2-kz))$, a direct calculation gives
\begin{equation}\label{eq:R-A}
|\psi_R(e^{i\theta})-1|^2=(\sqrt2-1)^2
\frac{\sqrt2+x}{\sqrt2-x}.
\end{equation}
By \eqref{eq:lim-rationalization-bound},
\[
\left|\sqrt{\psi_R(e^{i\theta})}-1\right|
\ge
\frac{|\psi_R(e^{i\theta})-1|}
{1+\sqrt{|\psi_R(e^{i\theta})|}}=\frac{|\psi_R(e^{i\theta})-1|}{1+V(x)^{1/4}}.
\]
Define
\[
H(x)
=
\frac{|\psi_R(e^{i\theta})-1|}{1+V(x)^{1/4}},
\qquad -1\le x\le1.
\]
We show that $H$ is strictly increasing on $[-1,1]$.
From \eqref{eq:R-A},
\[
\frac{d}{dx}\log|\psi_R(e^{i\theta})-1|
=
\frac{\sqrt2}{2-x^2},
\]
while \eqref{eq:R-V} gives
\[
\frac{d}{dx}\log V(x)^{1/4}
=
\frac14\left(
\frac{2x}{1+x^2}
+
\frac{1}{\sqrt2-x}
\right).
\]
Set
\[
B(x)
=
\frac14\left(
\frac{2x}{1+x^2}
+
\frac{1}{\sqrt2-x}
\right).
\]
Then
\[
\frac{H'(x)}{H(x)}
=
\frac{\sqrt2}{2-x^2}
-
\frac{V(x)^{1/4}}
{1+V(x)^{1/4}}\,B(x).
\]
If $B(x)<0$, then $H'(x)>0$. If $B(x)\ge0$, then
\[
\frac{H'(x)}{H(x)}
>
\frac{\sqrt2}{2-x^2}-B(x).
\]
A simplification gives
\[
\frac{\sqrt2}{2-x^2}-B(x)
=
-\frac{P(x)}
{4(\sqrt2-x)(2-x^2)(1+x^2)},
\]
where $P(x)=x^4+2\sqrt2\,x^3-11x^2+8\sqrt2\,x-6$. For $-1\le x\le0$, using
$x^4\le x^2$, $2\sqrt2\,x^3\le0$, and
$8\sqrt2\,x\le0$, we obtain
$P(x)\le-10x^2-6<0$.
For $0\le x\le1$, using $x^4\le x^2$ and $x^3\le x$, we obtain
\[
P(x)
\le
-10x^2+10\sqrt2\,x-6
=
-10\left(x-\frac1{\sqrt2}\right)^2-1<0.
\]
Thus $H'(x)>0$ throughout $[-1,1]$, and hence $H$ is strictly
increasing.
Therefore,
$
\left|\sqrt{\psi_R(e^{i\theta})}-1\right|
\ge
H(x)
\ge
H(-1).
$
At $x=-1$, corresponding to $z=-1$, we have
$\psi_R(-1)=2\sqrt2-2>0$. Hence equality holds in
\eqref{eq:lim-rationalization-bound}, and
$
H(-1)
=
\left|\sqrt{\psi_R(-1)}-1\right|
=
1-\sqrt{2\sqrt2-2}.
$
Consequently,
\[
\mathcal R(\varphi_{\mathrm{Lim}},\psi_R)
=
\sqrt2\left(1-\sqrt{2\sqrt2-2}\right).
\]

The minimum is attained at $z_0=-1$. Let
$r_0=\sqrt2\left(1-\sqrt{2\sqrt2-2}\right)$. Then
$
\varphi_{\mathrm{Lim}}(-r_0)
=
\left(1-{r_0}/{\sqrt2}\right)^2
=
2\sqrt2-2
=
\psi_R(-1).
$
Hence the radius $r_0$ is sharp.
\end{proof}

Next, for the exponential function $\psi_e(z)=e^z$,
the normalized square root satisfies $\sqrt{e^z}=e^{z/2}$
on $\overline{\mathbb D}$.
The following theorem determines the sharp
$\mathcal P_{\psi_e}$-radius of
$\mathcal P_{\varphi_{\mathrm{Lim}}}$
by reducing the boundary minimum to a real-variable
monotonicity problem.

\begin{theorem}
The $\mathcal{P}_{\psi_e}$-radius of the class $\mathcal{P}_{\varphi_{\mathrm{Lim}}}$ is
\[
\mathcal{R}(\varphi_{\mathrm{Lim}},\psi_e)=\sqrt2\left(1- \frac{1}{\sqrt e}\right) \approx0.55645.
\]
\end{theorem}

\begin{proof}
Let $\psi_e(z) = e^z$. By \eqref{R3-formula},
\[
\mathcal{R}(\varphi_{\mathrm{Lim}},\psi_e)
=\sqrt2\min_{-\pi\le\theta\le\pi}\left|{\sqrt{\psi_e(e^{i\theta})}-1}\right|
\]
For $z=e^{i\theta}$, we have $\operatorname{Im}z=\sin \theta\in[-1,1]$.
Thus the principal argument of $e^z$ equals
$\operatorname{Im}z$, and consequently
$
\sqrt{e^z}=e^{z/2}.
$
Now the problem simplifies to minimizing the modulus $|\exp(e^{i\theta}/2)-1|$ for $\theta \in [-\pi, \pi]$.

Let $x+iy = e^{i\theta}/2$, which gives $x = (\cos \theta)/2$ and $y = (\sin \theta)/2$. This bounds the variables to the circle $x^2+y^2 = 1/4$ with $x \in [-1/2, 1/2]$. $|e^{x+iy}-1|^2 = e^{2x}+1-2e^x\cos y$. Because this expression is even with respect to $y$, we may assume $y \ge 0$ and substitute $y = \sqrt{1/4-x^2}$. It then suffices to minimize the function
\[
F(x) = e^{2x}+1-2e^x\cos\left(\sqrt{\frac{1}{4}-x^2}\right), \quad x \in \left[ \frac{-1}{2}, \frac 12\right] 
\]
Differentiating $F(x)$ for $x \in (-1/2, 1/2)$ yields
\[
F'(x) = 2e^x\left(e^x - \cos y - x\frac{\sin y}{y}\right),\qquad
y=\sqrt{\frac14-x^2}.
\]
For $0 \le x < 1/2$, we apply the standard inequalities $e^x \ge 1+x$, $\cos y \le 1$, and $(\sin y)/y \le 1$. Substituting these yields $e^x - \cos y - (x \sin y)/y \ge (1+x) - 1 - x = 0$. For $0\le x<1/2$, we have $0<y\le1/2$ and hence
$\cos y<1$. Therefore,
\[
e^x-\cos y-x\frac{\sin y}{y}
\ge
1-\cos y
>0.
\]
For the left half of the interval, $-1/2 < x < 0$, we set $a = -x$ so that $a \in (0, 1/2)$. The parenthetical expression becomes $e^{-a} - \cos y + a(\sin y)/y$. By employing the elementary Taylor estimates $e^{-a} \ge 1-a$, $\cos y \le 1 - y^2/2 + y^4/24$, and $(\sin y)/y \ge 1 - y^2/6$, we obtain the lower bound:
\[
e^{-a} - \cos y + a\frac{\sin y}{y} \ge (1-a) - \left(1 - \frac{y^2}{2} + \frac{y^4}{24}\right) + a\left(1 - \frac{y^2}{6}\right) = y^2\left(\frac{1}{2} - \frac{a}{6} - \frac{y^2}{24}\right).
\]
Given the domain restrictions $a < 1/2$ and $y^2 < 1/4$, this lower bound is strictly greater than $y^2(13/32) > 0$.

Since $F'(x) > 0$ across the open interval $(-1/2, 1/2)$, the function $F$ is strictly monotonically increasing on $[-1/2, 1/2]$. Therefore, its absolute minimum occurs at the left endpoint $x = -1/2$, which corresponds to $\cos \theta = -1$, or $\theta = \pm \pi$. Evaluating the original modulus at $z = -1$ yields the minimum value:
\[
\mathcal{R}(\varphi_{\mathrm{Lim}},\psi_e) = \sqrt{2}\left| \frac{1}{\sqrt{e}} - 1 \right| = \sqrt{2}\left(1 - \frac{1}{\sqrt{e}}\right).
\]
Since the minimum is attained at $z_0=-1$, we have
$\psi_e(z_0)=e^{-1}$.
Let $r_0=\sqrt2\left(1-1/\sqrt e\right).$
Then
$
\varphi_{\mathrm{Lim}}(-r_0)
=\left(1-r_0/\sqrt2\right)^2
= e^{-1}
=
\psi_e(z_0).
$
Hence the radius is sharp.
\end{proof}

We next consider the three-leaf function
$\psi_{3L}(z)=1+4z/5+z^4/5$.
The boundary minimization is reduced to a polynomial positivity
problem on $[-1,1]$, which yields the sharp
$\mathcal P_{\psi_{3L}}$-radius of
$\mathcal P_{\varphi_{\mathrm{Lim}}}$.
\begin{theorem}
The $\mathcal P_{\psi_{3L}}$-radius of the class
$\mathcal P_{\varphi_{\mathrm{Lim}}}$ is
\[
\mathcal R(\varphi_{\mathrm{Lim}},\psi_{3L})
=
\sqrt{\frac{24-6\sqrt{15}}{5}}
\approx 0.390410.
\]
\end{theorem}

\begin{proof}

For $\psi_{3L}(z)=1+4z/5+z^4/5$, it follows from
\eqref{R3-formula} that
\[
\mathcal R(\varphi_{\mathrm{Lim}},\psi_{3L})
=
\sqrt2\min_{|z|=1}
\left|\sqrt{\psi_{3L}(z)}-1\right|.
\]
Set $z=e^{it}$, $-\pi\le t\le\pi$, and write
$w=\psi_{3L}(e^{it})=u+iv$. Let
$
D(t)=\left|\sqrt{w}-1\right|^2.
$
Since
$\operatorname{Re}\sqrt{w}=\sqrt{(|w|+u)/2}$, we have
\begin{equation}\label{eq:3L-D}
D(t)
=
|w|+1-\sqrt{2(|w|+u)}.
\end{equation}
Set $x=\cos t\in[-1,1]$. A direct calculation gives
\[
u(x)=\frac{8x^4-8x^2+4x+6}{5},
\qquad
V(x):=|w|^2
=
\frac4{25}
\left(20x^4+8x^3-20x^2+4x+13\right),
\]
and
$
|w-1|^2
=(1/25)(32x^3-24x+17).
$
Put
$
d=(3/5)(4-\sqrt{15}).
$
We show that $D(t)\ge d$ for all $t$. Write
$r=|w|=\sqrt{V(x)}$. From \eqref{eq:3L-D}, the inequality
$D(t)\ge d$ is equivalent to
$
2(r+u)\le (r+1-d)^2,
$
since $r+1-d>0$. Using
$r^2-2u+1=|w-1|^2$, this reduces to
$L(x)\ge2dr$, where
\[
L(x)
=
|w-1|^2+d^2-2d
=
\frac{2}{25}
\left(16x^3-12x+88-21\sqrt{15}\right).
\]
Since $16x^3-12x\ge-4$ on $[-1,1]$, we obtain
\[
L(x)
\ge
\frac{42}{25}(4-\sqrt{15})>0.
\]
Thus, it is sufficient to square the inequality $L(x)\ge2dr$.
A direct simplification gives
\begin{equation}\label{eq:3L-factor}
L(x)^2-4d^2V(x)
=
\frac4{625}(2x-1)^2P(x),
\end{equation}
where
$
P(x)=64x^4+64x^3
+(1440\sqrt{15}-5628)x^2
+(1848\sqrt{15}-7172)x
+48\sqrt{15}-149.
$
To prove that $P(x)>0$ on $[-1,1]$, write
$
P(x)=45(31-8\sqrt{15})+64(x+1)C(x),
$
where $C(x)=x^3-ax+b$ with
\[
a=\frac{1407-360\sqrt{15}}{16},
\qquad
b=\frac{51\sqrt{15}-193}{8}.
\]
Here $0<a<1$, $b>1/2$, and $b+a-1>0$.
For $-1\le x\le0$, writing $x=-y$ with $0\le y\le1$, we have
$C(x)=b+ay-y^3\ge b+(a-1)y \ge b+a-1 >0.$
For $0\le x\le1$, the minimum of $C(x)$ occurs at
$x=\sqrt{a/3}$, and hence
\[
C(x)
\ge
b-\frac{2a^{3/2}}{3\sqrt3}
>
\frac12-\frac{2}{3\sqrt3}
>0.
\]
Thus $P(x)>0$ throughout $[-1,1]$.
It follows from \eqref{eq:3L-factor} that
$L(x)^2\ge4d^2V(x)$, and hence $D(t)\ge d$.
Equality occurs when $2x-1=0$, that is, when
$x=1/2$, corresponding to $t=\pm\pi/3$.
At $x=1/2$, we have
$u=13/10$, $V=49/25$, and $|w|=7/5$. Therefore,
\[
D\left(\pm\frac{\pi}{3}\right)
=
\frac75+1
-
\sqrt{
2\left(\frac75+\frac{13}{10}\right)
}
=
\frac35(4-\sqrt{15}).
\]
Consequently,
\[
\mathcal R(\varphi_{\mathrm{Lim}},\psi_{3L})^2
=
2D\left(\pm\frac{\pi}{3}\right)
=
\frac65(4-\sqrt{15}),
\]
and therefore
\[
\mathcal R(\varphi_{\mathrm{Lim}},\psi_{3L})
=
\sqrt{\frac{24-6\sqrt{15}}{5}}.
\]
The minimum is attained at
$z_0=e^{i\pi/3}$ and its conjugate. Let
$
w_0
=
\sqrt2\left(
\sqrt{\psi_{3L}(z_0)}-1
\right).
$
Then
$|w_0|=\mathcal R(\varphi_{\mathrm{Lim}},\psi_{3L})$ and
$
\varphi_{\mathrm{Lim}}(w_0)
=\psi_{3L}(z_0)
\in
\partial\psi_{3L}(\mathbb D).
$
Hence the radius is sharp.
\end{proof}

Finally, we consider the modified sigmoid function
\[
\psi_{SG}(z)=\frac{2}{1+e^{-z}},
\]
whose associated Ma--Minda class was introduced and studied by
Goel and Kumar~\cite{MR4044913}. We use the relevant disk inclusion,
together with a direct boundary calculation, to determine the sharp
radius for the lima\c{c}on-associated class.

\begin{theorem}
The $\mathcal P_{\psi_{SG}}$-radius of the class
$\mathcal P_{\varphi_{\mathrm{Lim}}}$ is
\[
\mathcal R(\varphi_{\mathrm{Lim}},\psi_{SG})
=
\sqrt2\left(\sqrt{\frac{2e}{e+1}}-1\right)
\approx0.295826.
\]
\end{theorem}

\begin{proof}
For
$
\psi_{SG}(z)=2/(1+e^{-z}),
$
it follows from \eqref{R3-formula} that
\[
\mathcal R(\varphi_{\mathrm{Lim}},\psi_{SG})
=
\sqrt2\min_{|z|=1}
\left|\sqrt{\psi_{SG}(z)}-1\right|.
\]
Since
$\psi_{SG}(z)-1=\tanh(z/2)$,
we first determine the minimum of
$|\psi_{SG}(z)-1|$ on $|z|=1$.
Let $z=x+iy$, where $x^2+y^2=1$. Then
\[
\left|\tanh\frac z2\right|^2
=
\frac{\cosh x-\cos y}{\cosh x+\cos y}.
\]
By symmetry, we may assume $0\le x\le1$ and
$y=\sqrt{1-x^2}$. Define
\[
Q(x)
=
\frac{\cosh x}{\cos\sqrt{1-x^2}}.
\]
For $0<x<1$,
\[
\frac{Q'(x)}{Q(x)}
=
\tanh x-\frac{x}{y}\tan y<0,
\]
since $\tanh x<x$ and $\tan y>y$. Hence $Q$ is decreasing on
$[0,1]$, so that
$
Q(x)\ge Q(1)=\cosh1.
$
Consequently,
\[
\left|\tanh\frac z2\right|^2
=
\frac{Q(x)-1}{Q(x)+1}
\ge
\frac{\cosh1-1}{\cosh1+1}
=
\left(\frac{e-1}{e+1}\right)^2.
\]
Thus, with
\[
a=\frac{e-1}{e+1},
\]
we have
$
|\psi_{SG}(z)-1|\ge a,
$
with equality at $z=\pm1$.
Let
$
A=|\psi_{SG}(z)-1|.
$
By \eqref{eq:lim-A-bound},
\[
\left|\sqrt{\psi_{SG}(z)}-1\right|
\ge
\sqrt{1+A}-1.
\]
Since $A\ge a$ and $s\mapsto\sqrt{1+s}-1$ is increasing for
$s>0$, it follows that
\[
\left|\sqrt{\psi_{SG}(z)}-1\right|
\ge
\sqrt{1+a}-1
=
\sqrt{\frac{2e}{e+1}}-1.
\]
At $z=1$,
\[
\psi_{SG}(1)
=
\frac{2e}{e+1}
=
1+a>0,
\]
and equality holds throughout. Therefore,
\[
\min_{|z|=1}
\left|\sqrt{\psi_{SG}(z)}-1\right|
=
\sqrt{\frac{2e}{e+1}}-1,
\]
and hence
\[
\mathcal R(\varphi_{\mathrm{Lim}},\psi_{SG})
=
\sqrt2\left(
\sqrt{\frac{2e}{e+1}}-1
\right).
\]
Let
\[
r_0
=
\sqrt2\left(
\sqrt{\frac{2e}{e+1}}-1
\right).
\]
Since
$
\varphi_{\mathrm{Lim}}(r_0)
=2e/(e+1)
=\psi_{SG}(1)
\in\partial\psi_{SG}(\mathbb D).
$
Hence the radius is sharp.
\end{proof}

 We next consider the function
 $\psi_{\mathrm{car}}(z)=1+z+z^2/2$.
 A similar rationalization reduces the required estimate
 to the monotonicity of a real-valued function.
 The following theorem determines the sharp
 $\mathcal P_{\psi_{\mathrm{car}}}$-radius of
 $\mathcal P_{\varphi_{\mathrm{Lim}}}$.
 
 The following radius was obtained in
 \cite[Theorem~4.3(f)]{MR4303808}.
 We provide a proof through the general subordination radius formula.
 \begin{theorem}
 The $\mathcal{P}_{\psi_{\mathrm{car}}}$-radius of the class
 $\mathcal{P}_{\varphi_{\mathrm{Lim}}}$ is
 \[
 \mathcal{R}(\varphi_{\mathrm{Lim}},\psi_{\mathrm{car}})
 =
 \sqrt2-1\approx0.414214.
 \]
 \end{theorem}
 
 \begin{proof}
 Let $\psi_{\mathrm{car}}(z)=1+z+z^2/2$. By \eqref{R3-formula},
 \[
 \mathcal{R}(\varphi_{\mathrm{Lim}},\psi_{\mathrm{car}})
 =
 \sqrt2\min_{|z|=1}
 \left|\sqrt{\psi_{\mathrm{car}}(z)}-1\right|.
 \]
 Rationalizing the expression gives
 \begin{equation}\label{lim-car-eq1}
 \left|\sqrt{\psi_{\mathrm{car}}(z)}-1\right|
 =
 \frac{|\psi_{\mathrm{car}}(z)-1|}
 {\left|\sqrt{\psi_{\mathrm{car}}(z)}+1\right|}.
 \end{equation}
 For $z=e^{i\theta}$, set $x=\cos\theta\in[-1,1]$. We first estimate the numerator in terms of $x$.
 \[
 |\psi_{\mathrm{car}}(z)-1|
 = \frac12|z|\left| z+2 \right|
 =\frac12\sqrt{5+4x}.
 \]
 To estimate the denominator, Rewriting the function as
 $\psi_{\mathrm{car}}(z)=(1/2)\left(1+(z+1)^2\right)$, the triangle inequality gives
 \[
 |\psi_{\mathrm{car}}(z)|
 \le
 \frac{1+|z+1|^2}{2}
 =
 \frac{3+2x}{2}.
 \]
 Therefore, the triangle inequality gives
 \[
 \left|\sqrt{\psi_{\mathrm{car}}(z)}+1\right|
 \le
 \sqrt{|\psi_{\mathrm{car}}(z)|}+1
 \le
 1+\sqrt{\frac{3+2x}{2}}.
 \]
 Substituting these estimates into \eqref{lim-car-eq1},
 we obtain
 \[
 \left|\sqrt{\psi_{\mathrm{car}}(z)}-1\right|
 \ge
 \frac{\sqrt{5+4x}}
 {2\left(1+\sqrt{(3+2x)/2}\right)}.
 \]
 Introducing $u=\sqrt{(3+2x)/2}$, where
 $u\in[1/\sqrt2,\sqrt{5/2}]$, and using $5+4x=4u^2-1$, the lower bound reduces to
 \[
 H(u)=\frac{\sqrt{4u^2-1}}{2(u+1)}.
 \]
 Differentiating,
 \[
 H'(u)
 =
 \frac{4u+1}
 {2(u+1)^2\sqrt{4u^2-1}}
 >0.
 \]
 Thus $H$ is strictly increasing, and its minimum occurs at
 $u=1/\sqrt2$, corresponding to $x=-1$. Hence
 \[
 H\left(\frac{1}{\sqrt2}\right)
 =
 \frac{1}{2+\sqrt2}
 =
 1-\frac{1}{\sqrt2}.
 \]
 At $z=-1$, we have $\psi_{\mathrm{car}}(-1)=1/2$, and hence
 $\left|\sqrt{\psi_{\mathrm{car}}(-1)}-1\right|=1-1/\sqrt2$.
 Thus the lower bound is attained at $z=-1$, and therefore
 \[
 \mathcal{R}(\varphi_{\mathrm{Lim}},\psi_{\mathrm{car}})
 =
 \sqrt2\left(1-\frac{1}{\sqrt2}\right)
 =
 \sqrt2-1.
 \]
 Since the minimum is attained at $z_0=-1$, we have
 $\psi_{\mathrm{car}}(z_0)=1/2$.
 Let $r_0=\sqrt2-1$. Then
 $\varphi_{\mathrm{Lim}}(-r_0)
 =(1-r_0/\sqrt2)^2
 =1/2
 =\psi_{\mathrm{car}}(z_0)$.
 Hence the radius is sharp.
 \end{proof}

\begin{figure}[H]
\centering
\begin{subfigure}{0.48\textwidth}
\centering
\begin{tikzpicture}
    \begin{axis}[
        axis equal,
        axis lines=center,
        xmin=-0.2, xmax=3.2,
        ymin=-2, ymax=2,
        xtick={0, 1.0, 2.0, 3.0},
        ytick={-1.0, -0.5, 0.5, 1.0},
        thick,
        scale=0.75
    ]
    \addplot[
        domain=-180:180,
        samples=400,
        smooth,
        color=blue,
        line width=1pt,
    ]
    ({1 + (4/3)*cos(x) + (2/3)*cos(2*x)},
     {(4/3)*sin(x) + (2/3)*sin(2*x)});    
     
    \def\R{sqrt(2)*(1 - 1/sqrt(3))}
    \addplot[
        red,
        domain=-180:180,
        samples=400,
        smooth,
        line width=1pt
    ]
    ({quadX(\R,x)}, {quadY(\R,x)});
    \end{axis}
\end{tikzpicture}
\caption{$\mathcal{R}(\varphi_{\mathrm{Lim}},\psi_C)=\sqrt2\left(1- \frac{1}{\sqrt 3}\right)$}
\end{subfigure}
\begin{subfigure}{0.48\textwidth}
\centering
\begin{tikzpicture}
    \begin{axis}[
        axis equal,
        axis lines=center,
        xmin=-0.2, xmax=2.8,
        ymin=-1.5, ymax=1.5,
        xtick={0, 1.0, 2.0},
        ytick={-1.0, -0.5, 0.5, 1.0},
        thick,
        scale=0.75
    ]
    \addplot[
        domain=-180:180,
        samples=400,
        smooth,
        color=blue,
        line width=1pt,
    ]
    ({1 + cos(x) - (1/3)*cos(3*x)},
     {sin(x) - (1/3)*sin(3*x)});    
     
    \def\R{sqrt(2)*(sqrt(5/3) - 1)}
    \addplot[
        red,
        domain=-180:180,
        samples=400,
        smooth,
        line width=1pt
    ]
    ({quadX(\R,x)}, {quadY(\R,x)});
    \end{axis}
\end{tikzpicture}
\caption{$\mathcal{R}(\varphi_{\mathrm{Lim}},\psi_{\mathrm{Ne}})=\sqrt2\left(\sqrt{\frac{5}{3}}-1\right)$}
\end{subfigure}
\caption{Sharp radii associated with the lima\c{c}on function
$\varphi_{\mathrm{Lim}}$ for $\psi_{C}$ and $\psi_{Ne}$.}
\label{p2f3}
\end{figure}
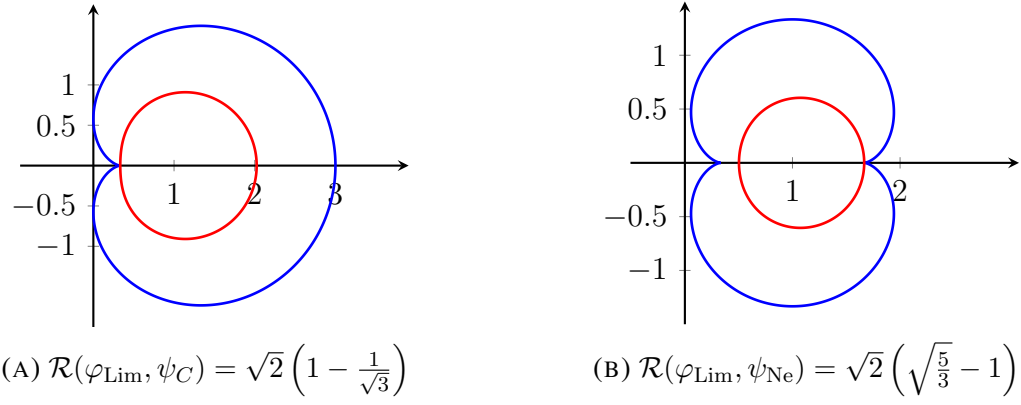

\begin{figure}[H]
\centering
\begin{subfigure}{0.48\textwidth}
\centering
\begin{tikzpicture}
    \begin{axis}[
        axis equal,
        axis lines=center,
        xmin=-0.2, xmax=3.0,
        ymin=-1.5, ymax=1.5,
        xtick={0, 1.0, 2.0},
        ytick={-1.0, -0.5, 0.5, 1.0},
        thick,
        scale=0.75
    ]
    \addplot[
        domain=-180:180,
        samples=400,
        smooth,
        color=blue,
        line width=1pt,
    ]
    ({exp(cos(x)) * cos(deg(sin(x)))},
     {exp(cos(x)) * sin(deg(sin(x)))});    
     
    \def\R{sqrt(2)*(1 - exp(-0.5))}
    \addplot[
        red,
        domain=-180:180,
        samples=400,
        smooth,
        line width=1pt
    ]
    ({quadX(\R,x)}, {quadY(\R,x)});
    \end{axis}
\end{tikzpicture}
\caption{$\mathcal{R}(\varphi_{\mathrm{Lim}},\psi_e)=\sqrt2\left(1- \frac{1}{\sqrt e}\right)$}
\end{subfigure}
\begin{subfigure}{0.48\textwidth}
\centering
\begin{tikzpicture}
    \begin{axis}[
        axis equal,
        axis lines=center,
        xmin=-0.2, xmax=2.2,
        ymin=-1.2, ymax=1.2,
        xtick={0, 0.5, 1.0, 1.5, 2.0},
        ytick={-1.0, -0.5, 0.5, 1.0},
        thick,
        scale=0.75
    ]
    \addplot[
        domain=-180:180,
        samples=400,
        smooth,
        color=blue,
        line width=1pt,
    ]
    ({1 + (4/5)*cos(x) + (1/5)*cos(4*x)},
     {(4/5)*sin(x) + (1/5)*sin(4*x)});    
     
    \def\R{sqrt(1.2*(4 - sqrt(15)))}
    \addplot[
        red,
        domain=-180:180,
        samples=400,
        smooth,
        line width=1pt
    ]
    ({quadX(\R,x)}, {quadY(\R,x)});
    \end{axis}
\end{tikzpicture}
\caption{$\mathcal R(\varphi_{\mathrm{Lim}},\psi_{3L})=\sqrt{\frac{24-6\sqrt{15}}{5}}$}
\end{subfigure}
\caption{Sharp radii associated with the lima\c{c}on function
$\varphi_{\mathrm{Lim}}$ for $\psi_e$ and $\psi_{3L}$.}
\label{p2f4}
\end{figure}
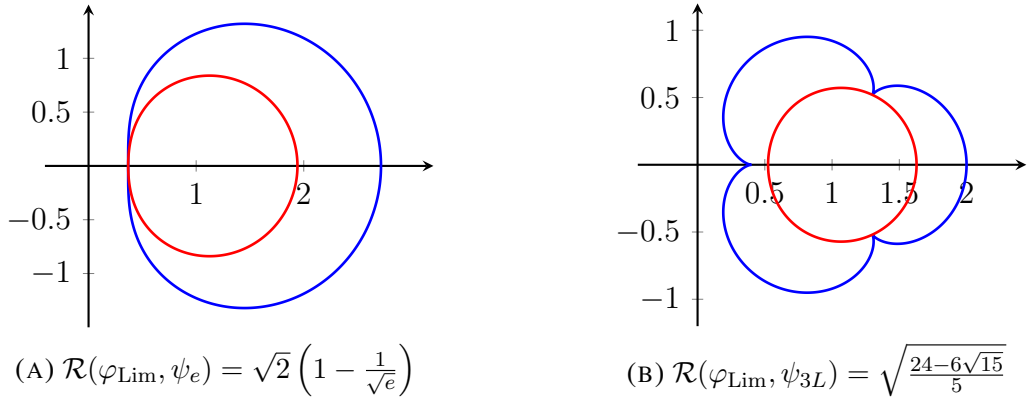

\begin{figure}[H]
\centering
\begin{subfigure}{0.48\textwidth}
\centering
\begin{tikzpicture}
    \begin{axis}[
        axis equal,
        axis lines=center,
        xmin=-0.2, xmax=2.5,
        ymin=-1.2, ymax=1.2,
        xtick={0, 0.5, 1.0, 1.5, 2.0},
        ytick={-1.0, -0.5, 0.5, 1.0},
        thick,
        scale=0.75,
        declare function={
            a = sqrt(2) + 1;
            denom(\t) = a^2 * (a^2 - 2*a*cos(\t) + 1);
            qX(\t) = 1 + (a^2 * cos(2*\t) + (a^3 - a) * cos(\t) - a^2) / denom(\t);
            qY(\t) = (a^2 * sin(2*\t) + (a^3 - a) * sin(\t)) / denom(\t);
        }
    ]
    \addplot[
        domain=-180:180,
        samples=400,
        smooth,
        color=blue,
        line width=1pt,
        variable=\t
    ]
    ({qX(\t)}, {qY(\t)});    
     
    \def\R{sqrt(2)*(1 - sqrt(2*sqrt(2) - 2))}
    \addplot[
        red,
        domain=-180:180,
        samples=400,
        smooth,
        line width=1pt,
        variable=\t
    ]
    ({quadX(\R,\t)}, {quadY(\R,\t)});
    \end{axis}
\end{tikzpicture}
\caption{$\mathcal R(\varphi_{\mathrm{Lim}},\psi_R)=\sqrt2\left(1-\sqrt{2\sqrt2-2}\right)$}
\end{subfigure}
\begin{subfigure}{0.48\textwidth}
\centering
\begin{tikzpicture}
    \begin{axis}[
        axis equal,
        axis lines=center,
        xmin=-0.2, xmax=2.5,
        ymin=-1.0, ymax=1.0,
        xtick={0, 0.5, 1.0, 1.5, 2.0},
        ytick={-1.0, -0.5, 0.5, 1.0},
        thick,
        scale=0.75,
        declare function={
            ec(\t) = exp(-cos(\t));
            es(\t) = deg(sin(\t));
            denomQ(\t) = 1 + 2*ec(\t)*cos(es(\t)) + ec(\t)^2;
            qX(\t) = 2*(1 + ec(\t)*cos(es(\t))) / denomQ(\t);
            qY(\t) = 2*(ec(\t)*sin(es(\t))) / denomQ(\t);
        }
    ]
    \addplot[
        domain=-180:180,
        samples=400,
        smooth,
        color=blue,
        line width=1pt,
        variable=\t
    ]
    ({qX(\t)}, {qY(\t)});    
     
    \def\R{sqrt(2)*(sqrt(2*exp(1)/(exp(1)+1)) - 1)}
    \addplot[
        red,
        domain=-180:180,
        samples=400,
        smooth,
        line width=1pt,
        variable=\t
    ]
    ({quadX(\R,\t)}, {quadY(\R,\t)});
    \end{axis}
\end{tikzpicture}
\caption{$\mathcal R(\varphi_{\mathrm{Lim}},\psi_{SG})=\sqrt2\left(\sqrt{\frac{2e}{e+1}}-1\right).$}
\end{subfigure}
\caption{Sharp radii associated with the lima\c{c}on function
$\varphi_{\mathrm{Lim}}$ for $\psi_R$ and $\psi_{SG}$.}
\label{p2f5}
\end{figure}
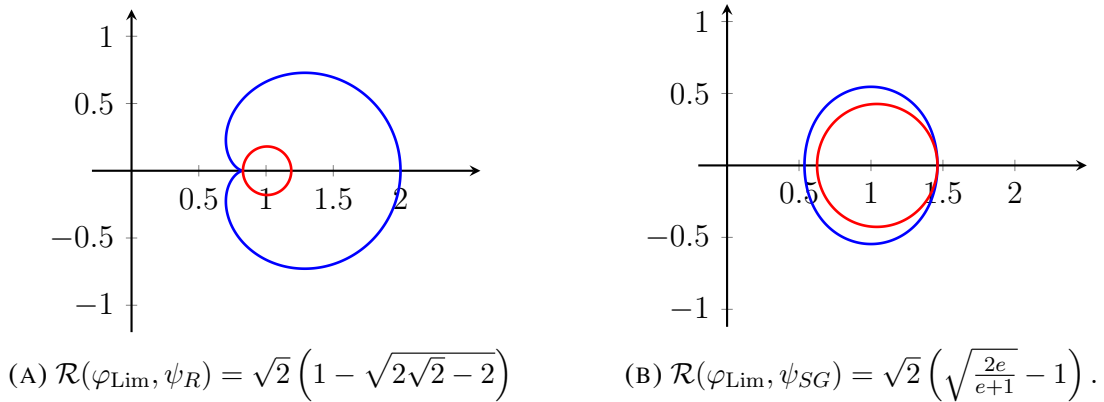

\section{Concluding Remarks}

We established a general formula for the sharp
$\mathcal P_{\psi}$-radius of $\mathcal P_{\varphi}$ under suitable
conditions on the inverse of $\varphi$. For Ma--Minda functions,
this formula also yields the corresponding radii for the associated
starlike and convex classes.

We applied the general result to the lemniscate-, lune-, and
lima\c{c}on-associated classes, obtaining sharp radius results for
the functions considered in Sections~3--5. The sharpness of each
result is established by identifying the boundary point at which
the corresponding minimum is attained. In particular, for the
lune and nephroid classes, we obtained the sharp radius
\[
\mathcal R(\varphi_{\leftmoon},\psi_{\mathrm{Ne}})
=\frac{8}{15},
\]
improving the previously known lower bound.

The results illustrate how the general radius formula can be used
to obtain sharp radii for different classes of normalized univalent
functions, provided the required inverse conditions hold and the
corresponding boundary minima can be determined.

\bigskip

\noindent
\textbf{Acknowledgement}: The first author is supported by a HTRA fellowship from NIT Tiruchirappalli.

\end{document}